\documentclass[11pt]{article}

\usepackage{epsfig}
\usepackage{graphicx}
\usepackage{amsbsy}
\usepackage{amsmath}
\usepackage{amsfonts}
\usepackage{amssymb}
\usepackage{textcomp}
\usepackage{hyperref}
\usepackage{aliascnt}

\newcommand{\mcm}[3]{\newcommand{#1}[#2]{{\ensuremath{#3}}}} 

\mcm{\tuple}{1}{\langle #1 \rangle}
\mcm{\name}{1}{\ulcorner #1 \urcorner}
\mcm{\Nbb}{0}{\mathbb{N}}
\mcm{\Zbb}{0}{\mathbb{Z}}
\mcm{\Rbb}{0}{\mathbb{R}}
\mcm{\Cbb}{0}{\mathbb{C}}
\mcm{\Qbb}{0}{\mathbb{Q}}
\mcm{\Acal}{0}{\cal A}
\mcm{\Bcal}{0}{\cal B}
\mcm{\Ccal}{0}{\cal C}
\mcm{\Dcal}{0}{\cal D}
\mcm{\Ecal}{0}{\cal E}
\mcm{\Fcal}{0}{\cal F}
\mcm{\Gcal}{0}{\cal G}
\mcm{\Hcal}{0}{\cal H}
\mcm{\Ical}{0}{\cal I}
\mcm{\Jcal}{0}{\cal J}
\mcm{\Kcal}{0}{\cal K}
\mcm{\Lcal}{0}{\cal L}
\mcm{\Mcal}{0}{\cal M}
\mcm{\Ncal}{0}{\cal N}
\mcm{\Ocal}{0}{{\cal O}}
\mcm{\Pcal}{0}{{\cal P}}
\mcm{\Qcal}{0}{{\cal Q}}
\mcm{\Rcal}{0}{{\cal R}}
\mcm{\Scal}{0}{{\cal S}}
\mcm{\Tcal}{0}{{\cal T}}
\mcm{\Ucal}{0}{{\cal U}}
\mcm{\Vcal}{0}{{\cal V}}
\mcm{\Xcal}{0}{{\cal X}}
\mcm{\Ycal}{0}{{\cal Y}}
\mcm{\Mfrak}{0}{\mathfrak M}

\DeclareMathOperator{\Cl}{Cl}
\mcm{\restric}{0}{\upharpoonright}
\mcm{\upset}{0}{\uparrow}
\mcm{\onto}{0}{\twoheadrightarrow}
\mcm{\smallNbb}{0}{{\small \mathbb{N}}}
\DeclareMathOperator{\preop}{op}
\mcm{\op}{0}{^{\preop}}

\newcommand{\se}{\subseteq}
{\begin{array}{c}
\setlength{\unitlength}{1em}}%
{\end{array}}

\usepackage{amsthm}

\newcommand{\theoremize}[2]{\newaliascnt{#1}{thm} \newtheorem{#1}[#1]{#2} \aliascntresetthe{#1}}

\theoremstyle{plain}
\newtheorem{thm}{Theorem}[section]
\theoremize{lem}{Lemma}
\theoremize{sublem}{Sublemma}
\theoremize{claim}{Claim}
\theoremize{rem}{Remark}
\theoremize{prop}{Proposition}
\theoremize{cor}{Corollary}
\theoremize{que}{Question}
\theoremize{oque}{Open Question}
\theoremize{con}{Conjecture}

\theoremstyle{definition}
\theoremize{dfn}{Definition}
\theoremize{eg}{Example}
\theoremize{exercise}{Exercise}
\theoremstyle{plain}

\usepackage{verbatim}
\usepackage{enumerate}
\usepackage[all]{xy}

\usepackage{subfig}
\theoremstyle{definition}
\theoremize{constr}{Construction}
\theoremstyle{plain}

\title{\scshape On the intersection conjecture for infinite trees of matroids}

\author{Nathan Bowler and Johannes Carmesin}

\newcommand{\sm}{\setminus}

\newcommand{\ct}{^\complement}

\DeclareMathOperator{\Sp}{Sp}
\newcommand{\gov}{\overline \gamma}

\begin{document}

\maketitle
\begin{abstract}
 Using a new technique, we prove a rich family of special cases of the matroid intersection conjecture. Roughly, we prove the conjecture for pairs of tame matroids which have a common decomposition by 2-separations into finite parts.
\end{abstract}

\section{Introduction}

In 2009, Aharoni and Berger \cite{AharoniBerger} proved the following infinite version of Menger's theorem: For every graph $G$ and $A,B\se V(G)$, there is a set of vertex-disjoint $A$-$B$-paths together with an $A$-$B$-separator consisting of precisely one vertex from each of these paths.
This had been conjectured by Erd\H{o}s in the 1960s.
In \cite{union2}, it was shown that 
Aharoni and Berger's theorem would follow from the following conjecture due to 
Nash-Williams \cite{Aharoni:Ziv:92}.

\begin{con}[The Matroid Intersection Conjecture]\label{int}
Any two matroids $M$ and $N$ on a 
common ground set $E$ have a common independent set $I$ admitting a partition $I
= J_M \cup J_N$ such that $\Cl_{M}(J_M) \cup \Cl_{N}(J_N) = E$.
\end{con}

This conjecture was intended as a generalisation of Edmonds' well known Intersection 
Theorem.
When Nash-Williams first made this conjecture in 1990, he only had finitary matroids in mind. But in 2008, Bruhn, Diestel, Kriesell, Pendavingh and Wollan \cite{matroid_axioms} introduced several equivalent axiomatisations for infinite matroids, providing a foundation on which a theory of infinite matroids with duality can be built. 
We shall work with a slightly better behaved subclass of infinite matroids, called tame matroids. This class includes all finitary matroids and all the other motivating examples of infinite matroids but  
is easier to work with than the class of infinite matroids in general
\cite{BC:wild_matroids}, \cite{BC:determinacy}, \cite{BC:ubiquity}, \cite{THINSUMS}, \cite{BC:rep_matroids},  
 \cite{BCC:graphic_matroids}.

The Matroid Intersection Conjecture, if true in this wider context, would immediately allow generalisations of the Aharoni-Berger Theorem to areas such as topological infinite graph theory. 

We showed in \cite{BC:PC} that the Matroid Intersection Conjecture is equivalent to a number of natural generalisations of theorems from finite graph theory and matroid theory, such as the Base Packing and Base Covering theorems. We also showed that it is equivalent to a conjecture which unifies Base Packing and Base Covering.

Suppose we have a pair of matroids $(M, N)$ on the same ground set $E$.
A  {\em packing} for this family consists of disjoint spanning sets $S^M$ and $S^N$ for $M$ and $N$ respectively. 
Similarly, a {\em covering} consists of independent sets $I^M$ and $I^N$ for $M$ and $N$ respectively whose union is the whole edge set.
Our new conjecture says that 
 that the ground set can be partitioned into a part which is ``dense'', in the sense that it has a packing, and a part which is ``sparse'', in the sense that it has a covering:

\begin{con}[The Packing/Covering Conjecture]\label{pc}
Let $M$ and $N$ be tame matroids on the same ground set $E$. Then $E$ admits a partition $E=P\dot\cup Q$
such that $(M \restric_P, N \restric_P)$  has a packing and $(M.Q, N.Q)$ has a covering.
\end{con}

Here $M\restric_P$ is the restriction of $M$ to $P$ and $M.Q$ is the contraction of $M$ onto $Q$.
Note that if $(M \restric_P, N \restric_P)$ has a packing, then $(M.P, N.P)$ has a packing, so we get a stronger statement by taking the restriction here. Similarly, we get a stronger statement by contracting, rather than restricting, to $Q$.

This conjecture is known to be true for $M$ and $N$ both finitary\footnote{This does not imply the Matroid Intersection Conjecture for pairs of finitary matroids, but rather for pairs of matroids in which one is finitary and the other cofinitary. 
This is because Packing/Covering for $M$ and $N$ is equivalent to Intersection for $M$ and the dual of $N$. Thus for our approach it is essential to work with a notion of infinite matroids which is closed under duality.
}, or even just nearly finitary~\cite{union2}, or when $M$ is finitary and $N$ is a countable direct sum of matroids whose duals are of finite rank~\cite{Aharoni:Ziv:92}, or when each of $M$ and $N$ has only countably many circuits~\cite{BC:PC}.
Apart from these very special cases, the conjecture has so far remained wide open, and seems to be difficult even for very simple classes of examples.

A natural approach to the investigation of this conjecture is to work out what it implies for some elementary examples. In this paper, we employ a new technique to resolve the conjecture for one rich class of such examples, which we will introduce in the next section.
Roughly, we consider pairs of tame matroids which have a common decomposition by 2-separations into finite parts.

\section{Overview of the results}

Let $G$ be a locally finite graph. We denote the finite cycle matroid of $G$ by $M_{FC}(G)$ and the topological cycle matroid of $G$ (see \cite{RD:HB:graphmatroids}) by $M_{TC}(G)$,
whose circuits, called \emph{topological circuits}, are the edge sets of topological circles in the topological space given by the graph together with its ends. The Packing/Covering Conjecture is known to hold for the pair $(M_{FC}(G), M_{FC}(G))$ since both  matroids in this pair are finitary. Since the conjecture is self-dual, it also holds for the pair $(M_{TC}(G),M_{TC}(G))$, in which both matroids are cofinitary. But what about the pair $(M_{FC}(G), M_{TC}(G))$, in which the first matroid is finitary and the second cofinitary?

One approach to this question is to consider a tree-decomposition of $G$ into finite parts.
We might hope to stick together packings and coverings at the finite decomposition-parts to a packing and a covering for the whole graph.
This approach seems most hopeful if the adhesion, the largest size of an intersection between adjacent parts, is bounded. We will show the following result along these lines:

\begin{thm}\label{cor1}
Let $G$ be a locally finite graph with a tree-decomposition into finite parts of adhesion at most 2. Then the pair $(M_{FC}(G), M_{TC}(G))$
satisfies the Packing/Covering Conjecture.
\end{thm}

The matroids $M_{FC}(G)$ and $M_{TC}(G)$ are not the only cycle matroids associated to the graph $G$. There is a large range of tame matroids sitting between these two. We say a matroid $M$
on $E(G)$ is a $G$-matroid if it is tame and each finite circuit of $G$ is a circuit of $M$ and each
 finite bond of $G$ is a cocircuit of $M$. For example, $M_{FC}(G)$ and $M_{TC}(G)$ are $G$-matroids. In \cite{BC:ubiquity}, we showed that every $G$-matroid arises 
from some subset $\Psi$ of the set of ends of $G$, in the sense that the circuits of $M$ are precisely the topological circuits of $G$ that use only ends from $\Psi$.
In this case, we denote $M$ by $M_\Psi(G)$. For example, $M_{FC}(G)=M_{\emptyset}(G)$ and $M_{TC}(G)=M_{\Omega(G)}(G)$, where $\Omega(G)$ is the set of all ends of $G$.
 
Not every set $\Psi$ gives rise to a matroid in this way, but if $\Psi$ is pleasant enough, then we do get a matroid \cite{BC:determinacy}.

The way pleasantness is measured here has to do with Determinacy of Sets.
Determinacy of sets is usually defined using games.
Let $\Phi\se A^\Nbb$ for some set $A$. The  $\Phi$-game $\Gcal(\Phi)$ is the following game between two players
which has one move for every natural number.
In each odd move the first player chooses an element of $A$ and in each even move the second player chooses such an element. The first player wins if and only if the sequence they generate between them is in $\Phi$.
The set $\Phi$ is \emph{determined} if one player has a winning strategy.
The question which sets are determined has been investigated in detail by set theorists \cite{determinacy_survey}:
The statement that all subsets $\Phi\se A^\Nbb$ with $A$ countable are determined is called
\emph{the Axiom of Determinacy}, and is sometimes taken as an alternative  to the Axiom of Choice. Indeed, if one assumes the Axiom of Determinacy instead of the Axiom of Choice, every set of real numbers becomes Lebesgue measurable \cite{determinacy_to_measurability}.
A deep result in this area says that if $\Phi$ is Borel (in the product topology), then it is determined \cite{Martin_borel}. 

In \cite{BC:determinacy}, we prove that if $f^{-1}\Psi$ is determined for each $f$ in some fixed collection of continuous maps then there is a matroid $M_{\Psi}(G)$. In particular, if $\Psi$ is Borel then there is such a matroid. Using the determinacy of a different collection of games, in this paper we will prove the following:

\begin{thm}\label{cor2}
Let $G$ be a locally finite graph with a tree-decomposition into finite parts of adhesion at most 2, and let $\Psi_1$ and $\Psi_2$ be Borel sets of ends of $G$. Then the pair $(M_{\Psi_1}(G), M_{\Psi_2}(G))$
satisfies the Packing/Covering Conjecture.
\end{thm}

We will in fact prove this for the more general family of matroids introduced in \cite{BC:determinacy}, namely those obtained by sticking together finite matroids using 2-sums along a tree structure (see \autoref{mainthm}). Combining this with the basic structural theory of tame matroids developed in \cite{ADP:decomposition} and \cite{BC:ubiquity}, according to which any matroid has a canonical decomposition over its 2-separations into such a tree structure, we have the first beginnings of a structural attack on the Packing/Covering conjecture. First, assuming the axiom of determinacy, we can show that if $M$ is a connected matroid all of whose 3-connected minors are finite then $(M, M)$ satisfies Packing/Covering. Without the axiom of determinacy we do not have this result in general, but we still have it for well-enough behaved matroids. 

Allowing the matroids at the nodes of the tree to be arbitrary finite matroids gives us a lot of freedom: for example, taking all the matroids in the tree to be large and uniform we 
get a natural combinatorial statement about when an infinite system of committees can make a decision. This translation is explained in \cite{C:committees}, where a slightly more general version of this statement is proved.

\section{Overview of the proof}\label{proof_strategy}

We now turn to a sketch of the proof of \autoref{cor2}. For reasons outlined in \cite{BC:PC},  it is enough to prove for pairs $(M, N)$ of matroids as in that theorem that for every edge $e$ of the ground set there is either a set $P$ containing $e$ such that $(M \restric_P, N \restric_P)$  has a packing (we call such a $P$ a wave) or a set $Q$ containing $e$ such that $(M.Q, N.Q)$ has a covering  (we call such a $Q$ a co-wave). 

We therefore look at an example of how such a wave $P$ can interact with a common 2-separation of $M$ and $N$:
Assume $M=M_1\oplus_2 M_2$ and $N=N_1\oplus_2 N_2$ and 
$E(M_1)=E(N_1)$ and $E(M_2)=E(N_2)$.\footnote{Here $\oplus_2$ denotes the 2-sum.}
We assume that $e\in E(M_1)$ and call the gluing edge $f$.

Now suppose that in $(M_1\sm f, N_1/f)$ there is a wave $P_1$ containing $e$ with spanning sets
$S^M$ and $S^N$, and in $(M_2,N_2)$ there is a wave $P_2$ avoiding $f$ with spanning sets
$T^M$ and $T^N$ such that $f$ is in the $N_2$-span of $T^N$.
We can stick together these two waves to give a wave $P=P_1\cup P_2$ in $(M,N)$ with spanning sets $S^M\cup T^M$ and $S^N\cup T^N$. 
We imagine the wave $P_1$ as relying on a promise from $P_2$ that it will $N$-span the edge $f$.
This is one of the 6 ways, classified in \autoref{games}, in which a wave in $(M,N)$ can be built from waves in the two smaller pairs.

Our result relies on the determinacy of certain games. The first is called the \emph{Packing game}, and is played between two players, called Packer and Coverina: we think of Packer as trying to build a wave and Coverina as trying to stop him. At any point in the game, Packer has a partially built wave, together with a collection of promises on which this `partial wave' relies. Coverina is allowed to challenge one of these promises, at which point Packer must show that it can be fulfilled by giving a partial wave fulfilling it, which relies on further promises, which Coverina may in turn challenge, etc.

The game is designed to have the property that there is a wave containing $e$ if and only if Packer has a winning strategy in this game.
Similarly to the Packing game, we define a Covering game, where Coverina is trying to build a suitable co-wave and Packer is trying to prevent her from doing this.
These two games will be determined because $\Psi_1$ and $\Psi_2$ are Borel.
Thus, it suffices to show that we cannot have both a winning strategy for Packer 
in the Covering game and a winning strategy for Coverina
in the Packing game.

We show that if there were such strategies then it would be possible to in some sense play them off against each other, recursively producing infinite plays in each strategy one of which must be losing. Since the strategies were supposed to be winning, this gives our desired contradiction. In the recursive construction we can work locally within particular pairs of finite matroids. However, as often happens, the result about finite matroids which we need to apply is not quite the specialisation of our result to finite matroids. Instead we need a strengthening of the Packing/Covering theorem for finite matroids, explained in \autoref{lemma27} and \autoref{lemma17}. 

Although \autoref{cor2} implies \autoref{cor1}, we are not able to give a simpler proof of 
\autoref{cor1} that does not use these games. In fact, our proof of \autoref{cor1} already relies on the principle of $\mathbf{\Sigma^0_2}$-Determinacy, that is, determinacy of sets $\Phi$ which are countable unions of closed sets. 

The paper is organised as follows. After introducing some necessary background in \autoref{prelims}, we define the Packing and Covering games in \autoref{games}. We then reduce our result to the special case in which $\Psi_M$ and $\Psi_N$ partition the set of ends in \autoref{block}. This special case is proved in \autoref{main}, relying on some technical lemmas from \autoref{cases}.

\section{Preliminaries}\label{prelims}

Throughout, notation and terminology for graphs are those of~\cite{DiestelBook10}, and for matroids those of~\cite{Oxley,matroid_axioms}. We will rely on the following lemma from \cite{DiestelBook10}:

\begin{lem}[K\"onig's Infinity Lemma \cite{DiestelBook10}]\label{Infinity_Lemma}
Let $V_0,V_1,\ldots$ be an infinite sequence of disjoint non-empty finite sets, and let $G$ be a graph on their union. Assume that every vertex in $V_n$ with $n\geq 1$ has a neighbour in $V_{n-1}$. Then $G$ includes a ray $v_0v_1\ldots$ with $v_n\in V_n$ for all $n$.
\end{lem}

For any vertex $t$ of a rooted tree $T$ other than the root, $t^-$ is the unique neighbour of $t$ which is closer to the root.
Whenever, we have a rooted tree $T$, we will consider the edges to be directed towards the root. The terminal vertex of an edge $e$ is denoted by $t(e)$, and the initial vertex by $s(e)$. 
For a set $X$ of edges of such a tree, let $T_V(X)$ denote the set of terminal vertices, and 
$S_V(X)$ the set of starting vertices of edges in $X$.
For a set $F$ of edges, let $V(F)$ be the set of vertices incident with edges in $F$.
For a vertex set $Z$, we denote by $E(Z)$ the set consisting of those edges with both endvertices in $Z$.

By $\pi_1$ and $\pi_2$ we denote the two coordinate-projections for ordered pairs. 

As usual, we denote by $r(M)$ the rank of a finite matroid $M$.

A \emph{strategy for the first player} in a game $\Gcal$ is a set $\sigma$ of finite odd-length plays $P$
such that the following is true for all $P\in \sigma$:
Let $m$ be a move of the second player such that $P m$ is a legal play.
Then there is a unique move $m'$ of the first player such that $P m m'\in \sigma$.
Furthermore, we require that $\sigma$ is closed under \emph{2-truncation}, that is, 
for every $P\in \sigma$ there are some
$P'\in \sigma$ and moves $m$ and $m'$ of the second player and the first player, respectively, such that
$P' m m'=P$.

An infinite play \emph{belongs} to a strategy $\sigma$ for the first player if all its odd length finite initial plays are in $\sigma$.
A strategy for the first player is \emph{winning} if the first player wins in all infinite plays belonging to $\sigma$.
Similarly, one defines \emph{strategies} and \emph{winning strategies} for the second player.

\subsection{Waves and cowaves}

Let $(M,N)$ be a pair of matroids on the same ground set $E$.
 $(X,S^M,S^N)$ is a \emph{wave} for $(M,N)$ if $S^M$ is spanning in $M\restric_X$ and $S^N$ is spanning in $N\restric_X$. We will sometimes refer to the wave simply as $X$, leaving the other sets implicit. 
A \emph{hindrance} is a wave such that there is some $e\in X\sm (S^M\cup S^N)$. In these circumstances, we say that the wave \emph{focuses on $e$}.
We say an edge $e$ is \emph{$M$-spanned} by the wave $X$ if $e$ is in the $M$-span of $X$ but not in $X$ itself. For any wave $X$ for $(M,N)$ and set $C$ of edges, $X\sm C$ is a wave for $(M/C,N/C)$.

A \emph{cowave} for $(M,N)$ is a wave for $(M^*,N^*)$, that is, a triple $(Y,T^M,T^N)$ such that $T^M$ is cospanning in $M.Y$ and $T^N$ is cospanning in $N.Y$.
A cohindrance is a hindrance for $(M^*,N^*)$.
A cowave \emph{$M$-cospans $e$} if it $M^*$-spans $e$ when considered as a wave for $(M^*,N^*)$.

We will need the following lemmas about waves from \cite{BC:PC}:

We define a partial order on waves by $(X, S^M, S^N) \leq (Y, T^M, T^N)$ if and only if $X \subseteq Y$, $S^M \subseteq T^M$ and $S^N \subseteq T^N$. We say a wave is {\em maximal} when it is maximal with respect to this partial order.

\begin{lem}[{\cite[Lemma 4.3 and Corollary 4.5]{BC:PC}}]\label{maxwave}
There is a maximal wave, which covers every edge that is covered by any wave.
\end{lem}

\begin{lem}[{\cite[Lemma 4.4]{BC:PC}}]\label{joinwaves}
Let $(X, S^M, S^N)$ and $(Y, T^M, T^N)$ be waves for $(M, N)$.
Then $(X\cup Y, S^M \cup (T^M \sm X), S^N \cup (T^N \sm X))$ is a wave, which we denote $X \circ Y$.
\end{lem}

If $X$ is a hindrance focusing on $e$ then so is $X \circ Y$. If $e$ is $M$-spanned by $X$ and not contained in $Y$ then $e$ is $M$-spanned by $X \circ Y$.

\begin{lem}[{\cite[Lemma 4.7]{BC:PC}}]\label{contract_waves}
Let $(X, S^M,S^N)$ be a wave for a pair $(M,N)$ of matroids on the same ground set. 
Let $(Y, T^M,T^N)$ be a wave for the pair $(M/X,N/X)$. 
Then $(X \cup Y, S^M\cup T^M, S^N\cup T^N)$ is a wave for the family $(M,N)$.
\end{lem}

\begin{cor}\label{PCfromwave}
Let $(M, N)$ be a pair of matroids on the same ground set $E$. If for any set $X$ and any edge $e \in E \sm X$ there is either a wave in $(M/X, N/X)$ containing $e$ or a cohindrance in $(M/X, N/X)$ focusing on $e$ then $(M, N)$ satisfies the Packing/Covering Conjecture.
\end{cor}
\begin{proof}
 Let $X$ be a maximal wave. Then by \autoref{contract_waves} there is no nontrivial wave in $(M/X, N/X)$, so by assumption every edge in $E \sm X$ is at the focus of some cohindrance. So by the dual of \autoref{maxwave} there is a cowave for this pair whose underlying set is $E \setminus X$. This cowave, together with $X$, witnesses that $(M, N)$ satisfies the Packing/Covering Conjecture.
\end{proof}

Next we recall the concept of exchange chains as introduced in \cite{union1}.
For sets $I_M\in\Ical(M)$ and $I_N\in\Ical(N)$, and elements $x\in I_M\cup I_N$ and  $y\in E$, a tuple $Y=(y=y_0, \ldots, y_n=x)$ with $y_i \neq y_{i+1}$ for all $i$ is called an \emph{even $(I_M, I_N)$-exchange chain} (or \emph{even $(I_M, I_N)$-chain}) 
from $y$ to $x$
of \emph{length} $n$ if the following terms are satisfied. 

\begin{enumerate}[(X1)]
	\item For an even $i$, there exists a circuit $C_i$ of $M$ with $\{y_i, y_{i+1}\} \se C_i \subseteq I_M+y_i$.
	\item For an odd $i$, there exists a circuit $C_i$ of $N$ with $\{y_i, y_{i+1}\} \se C_i \subseteq I_N+y_i$.
\end{enumerate}

\noindent
If $n\geq 1$, then (X1) and (X2) imply that $y_0\notin I_M$ and that, starting with $y_1\in I_M\setminus I_N$, the elements $y_i$ alternate between $I_M\setminus I_N$ and $I_N\setminus I_M$; the single exception being $y_n$ which can lie in $I_M\cap I_N$.

By an \emph{odd exchange chain} (or \emph{odd chain}) we mean an even chain with the words `even' and `odd' interchanged in the definition.
Consequently, we say \emph{exchange chain} (or \emph{chain}) to refer to either of these notions.

\begin{lem}[{\cite[weakening of Lemma 2.5]{BC:PC}}]\label{runchains}
Let $(M,N)$ be a pair of matroids on the same ground set and let $B_M\in \Ical(M)$ and $B_N\in \Ical(N)$.
If there is a $(B_M,B_N)$-exchange chain from $z$ to $f$,
then there are sets $B_M'\in \Ical(M)$ and $B_N'\in \Ical(N)$ 
such that $B_M'\cup B_N'=B_M\cup B_N+z-f$.

Moreover, $\Cl_{M}B_M=\Cl_{M}B_M'$ and $\Cl_{N}B_N=\Cl_{N}B_N'$.
\end{lem}

\begin{lem}\label{chain3}
Let $(M,N)$ be a pair of matroids on a common ground set $E$, and let $f\in E$.
If there is a hindrance $(X,S^M,S^N)$ in $(M/f,N\sm f)$, then in $(M,N)$ either $X$ is a wave or there is a hindrance $X'\se X$.
\end{lem}

\begin{proof}
We may assume that $f$ is not a loop in $M$, and that $S^M$ and $S^N$ are bases of $(M/f)\restric_X$ and $(N\sm f)\restric_X$, respectively. 
Thus $S^M+f$ is $M$-independent. Let $z$ be in the focus of the hindrance. 
Let $X'\se X$ be the set of edges $y$ for which there is some $(S^M+f, S^N)$-chain
from $z$ to $y$.  
First we consider the case that $f\notin X'$.
Then $(X', S^M\cap X', S^N\cap X')$ is a hindrance focusing on $z$, which is the second outcome of the lemma.

Thus we may assume that there is an $(S^M+f, S^N)$-chain from $z$ to $f$.
Applying \autoref{runchains}, we get sets $J_M\in \Ical(M)$ and $J_N\in \Ical(N)$
such that $S^M\cup S^N+z=J_M\cup J_N$. Moreover, $J_M$ and $J_N$ span $X$ in $M$ and $N$, respectively.
Hence we get the first outcome: $(X, J_M,J_N)$ is a wave.
\end{proof}


\begin{lem}\label{chain2}
Let $(M,N)$ be a pair of matroids on the common ground set $E$, and let $e\in E$.
If there is a hindrance $(X,S^M,S^N)$ for $(M,N)$, 
then in $(M,N)$ either there is a hindrance focusing on $e$ or there is a hindrance that does not contain $e$.
\end{lem}

\begin{proof}
If $e$ is not in $S^M\cup S^N$, then we are done.
So we assume without loss of generality that $e\in S^M$.
Then $X-e$ is a hindrance for $(M/e,N\sm e)$.
By \autoref{chain3}, in $(M,N)$ either $X-e$ is a wave or there is a hindrance $X'\se X-e$.
As we are done in the later case, it suffices to show that $e$ is $M$-spanned and $N$-spanned by $X-e$. As $e\in S^M$, it is $N$-spanned by $S^N\se X-e$. If $e$ is not $M$-spanned by $X-e$, then
$(X-e,S^M-e,S^N)$ is a hindrance avoiding $e$, in which case we are also done.
\end{proof}

\begin{lem}\label{chain4}
Let $(M,N)$ be a pair of finite matroids on a common ground set $E$, and let $e\in E$. 
Then either there is a cohindrance focusing on $e$ or $e$ is contained in a wave.
\end{lem}

\begin{proof}
We assume that $e$ is not contained in any wave.
Let $X$ be a maximal wave. 
Let $Y$ be a maximal cowave for $(M/(X + e),N/(X + e))$.
Since Packing/Covering holds for finite matroids \cite{BC:PC},
we can apply it to the pair $(M/X\sm Y,N/X\sm Y)$.
By \autoref{contract_waves}, $E\sm (X\cup Y)$ does not include a wave, so is a cowave. 
It cannot be a wave, so it is a cohindrance. 
By the dual of \autoref{contract_waves}, any cohindrance for $(M/X\sm Y,N/X\sm Y)$ contains $e$.
So by the dual of \autoref{chain2}, we get a cohindrance focusing on $e$ for $(M/X\sm Y,N/X\sm Y)$, which gives rise to a cohindrance focusing on $e$ for $(M,N)$ by the dual of \autoref{contract_waves}.
\end{proof}

\begin{lem}\label{intermediate_lemma7}
Let $M$ and $N$ be two matroids on a common finite ground set $E$. Let $e,f\in E$ distinct.
Assume that every nonempty wave for $(M,N)$ contains $e$.
Then in $(M/f,N\sm f)$  either $E-f$ is a cowave or there is a hindrance focusing on $e$.
\end{lem}

\begin{proof}
We assume that  $E-f$ is not a cowave for $(M/f,N\sm f)$. 
Then by \autoref{maxwave}, there is an edge $g$ not in any cowave for $(M/f,N\sm f)$. By the dual of \autoref{chain4},
we get that there is a hindrance $(X,S^M,S^N)$ for $(M/f,N\sm f)$ focusing on $g$.
Now we apply \autoref{chain2} to $(X,S^M,S^N)$ and the edge $e$.
To show that there is a hindrance focusing on $e$, 
it suffices to show that there cannot be a hindrance $(X',{S'}^M,{S'}^N)$ with $e\notin X'$.
If there were, then by \autoref{chain3} we would get that $X'$ is a wave for $(M,N)$ or that there is a 
hindrance $X''\se X'$ for $(M,N)$. Both of these contradict the assumption that every wave contains $e$. Thus there is a hindrance focusing on $e$, which completes the proof.
\end{proof}

 \subsection{Trees of matroids}

\begin{dfn}
A {\em tree $\Tcal$ of matroids} consists of a tree $T$, together with a function $\overline{M}$ assigning to each node $t$ of $T$ a matroid $\overline{M}(t)$ on the ground set $E(t)$, such that for any two nodes $t$ and $t'$ of $T$, if $E(t) \cap E(t')$ is nonempty then $tt'$ is an edge of $T$.

For any edge $tt'$ of $T$ we set $E(tt') = E(t) \cap E(t')$. We also define the {\em ground set} of $\Tcal$ to be $E = E(\Tcal) = \left(\bigcup_{t \in V(T)} E(t)\right) \setminus \left(\bigcup_{tt' \in E(T)} E(tt')\right)$. 

We shall refer to the edges which appear in some $E(t)$ but not in $E$ as {\em dummy edges} of $\overline{M}(t)$: thus the set of such dummy edges is $\bigcup_{tt' \in E(T)} E(tt')$.
\end{dfn}

The idea is that the dummy edges are to be used only to give information about how the matroids are to be pasted together, but they will not be present in the final pasted matroid, which will have ground set $E(\Tcal)$. 
We will now consider a type of pasting corresponding to 2-sums. We will make use of some additional information to control the behaviour at infinity: a set $\Psi$ of ends of $T$. 

\begin{dfn}
A tree $\Tcal = (T, \overline{M})$ of matroids is {\em of overlap 1} if, for every edge $tt'$ of $T$, $|E(tt')| = 1$. In this case, we denote the unique element of $E(tt')$ by $e(tt')$. Given $F\se E(T)$, let $e``F=\{e(f)\mid f\in F\}$.

Given a tree of matroids of overlap 1 as above and a set $\Psi$ of ends of $T$, a {\em $\Psi$-pre-circuit} of $\Tcal$ consists of a connected subtree $S$ of $T$ together with a function $\overline{o}$ assigning to each vertex $t$ of $S$ a circuit of $\overline{M}(t)$, such that all ends of $S$ are in $\Psi$ and for any vertex $t$ of $S$ and any vertex $t'$ adjacent to $t$ in $T$, $e(tt') \in \overline{o}(t)$ if and only if $t' \in S$. The set of $\Psi$-pre-circuits is denoted $\overline\Ccal(\Tcal, \Psi)$. 

Any $\Psi$-pre-circuit $(S, \overline{o})$ has an {\em underlying set} $\underline{(S, \overline{o})} = E \cap \bigcup_{t \in V(S)} \overline{o}(t)$. Minimal nonempty subsets of $E$ arising in this way are called {\em $\Psi$-circuits} of $\Tcal$. The set of $\Psi$-circuits of $\Tcal$ is denoted $\Ccal(\Tcal, \Psi)$.
\end{dfn}

We shall rely on the following theorem and lemmas.

\begin{thm}[\cite{BC:determinacy}]\label{psimat}
 Let $\Tcal = (T, \overline{M})$ be a tree of matroids of overlap 1 and $\Psi$ a Borel set of ends of $T$, then there is a matroid $M_\Psi(T,\overline{M})$ whose circuits are the $\Psi$-circuits.
\end{thm}

\begin{lem}[\cite{BC:determinacy}]\label{precirc_is_scrawl}
 In the context of \autoref{psimat}, the underlying set $\underline{(S, \overline{o})}$ of any $\Psi$-precircuit $(S, \overline{o})$ is a union of $\Psi$-circuits.
\end{lem}

\begin{lem}[\cite{BC:determinacy}, Lemma 5.4]\label{ToMminor}
 Let $M_{\Psi}(T, \overline M)$ be a matroid as above and let $C$ and $D$ be disjoint subsets of its ground set. Then
$$M_{\Psi}(T, \overline M)/C\backslash D = M_{\Psi}(T, t \mapsto \overline M(t) / C \backslash D)$$
\end{lem}

\begin{dfn}\label{Tsub}
 If $T$ is a tree, and $tu$ is a (directed) edge of $T$, we take $T_{t \to u}$ to be the connected component of $T - t$ that contains $u$. If $\Tcal = (T, \overline{M})$ is a tree of matroids, we take $\Tcal_{t \to u}$ to be the tree of matroids $(T_{t \to u}, \overline{M} \restric_{T_{t \to u}})$.
\end{dfn} 

We will also need the following lemma, which allows us to choose families of compatible pre-circuits:

\begin{lem}\label{pick_circuits}
Let $T$ be a rooted tree with root $t_0$, let $\Tcal = (T, \overline M)$ be a tree of matroids of overlap 1 and let $\Psi$ be a Borel set of ends of $T$. Let $X$ be any subset of $E(M_{\Psi}(\Tcal))$, and let $U$ be the set of nodes $t$ of $T$ such that $e(t^-t)$ is spanned by $X \cap E(\Tcal_{t^- \to t})$ in $M_{\Psi}(\Tcal_{t^- \to t})$. Then there is a choice of a $\Psi$-precircuit $(S_t, \overline{o}_t)$ in $\Tcal_{t^- \to t}$ for each $t \in U$ witnessing this in the sense that $e(t^-t) \in \underline{(S_t, \overline{o}_t)} \subseteq X + e(t^-t)$ and such that for any nodes $u$, $v$ and $w$ with $w \in S_u \cap S_v$ we have $\overline{o}_u(w) = \overline{o}_v(w)$.
 \end{lem}
 \begin{proof}
 We denote the tree-order in $T$ by $\leq_T$.
 
 We construct the pre-circuits $(S_t, \overline{o}_t)$ recursively in the height of $t$ in $T$, so for each $n$ we choose all $(S_t, \overline{o}_t)$ with $t$ at height $n$ before choosing those with greater heights. When choosing $(S_u, \overline{o}_u)$, we first of all check whether there is some $t <_T u$ with $u \in S_t$. If so, we pick $t$ minimal with this property and let $S_u = S_t \cap T_{u^- \to u}$ and $\overline{o}_u(v) = \overline{o}_t(v)$ for each $v \in S_u$. Otherwise, we pick any $(S_u, \overline{o}_u)$ such that  $e(u^-u) \in \underline{(S_u, \overline{o}_u)} \subseteq X + e(u^-u)$: there is some such pre-circuit since $u \in U$. 
 
The only thing to check is that for any nodes $u$, $v$ and $w$ with $w \in S_u \cap S_v$ we have $\overline{o}_u(w) = \overline{o}_v(w)$. So suppose we have such $u$, $v$ and $w$. By construction, $u \leq_T w$ and $v \leq_T w$, so without loss of generality $u \leq_T v$. Since $u \leq_T v \leq_T w$ and both $u$ and $w$ are in $S_u$, we must also have $v \in S_u$. Let $t \leq_T u$ be minimal with $u \in S_t$. Then by construction, since $v \in S_u$ we also have $v \in S_t$. Further, $t$ is minimal with $v \in S_t$ since if there were $t' <_T t$ with $v \in S_{t'}$ we would also have $u \in S_{t'}$ (since $t' \leq_T u \leq_T v$), contradicting our choice of $t$. Thus $\overline{o}_u(w) = \overline{o}_t(w) = \overline{o}_v(w)$, as required.
 \end{proof}

Our main result will be the following:

\begin{thm}\label{mainthm}
 Let $(T, \overline M)$ and $(T, \overline N)$ be trees of matroids of overlap 1 such that for any $t \in V(T)$, the matroids $\overline M(t)$ and $\overline N(t)$ have the same finite ground set $E(t)$. Let $\Psi_M$ and $\Psi_N$ be Borel sets of ends of $T$. Then the pair $(M_{\Psi_M}(T, \overline M), M_{\Psi_N}(T, \overline N))$ of matroids satisfies the Packing/Covering Conjecture.
\end{thm}

We will go via the following special case of this theorem:

\begin{prop}\label{mainprop}
 \autoref{mainthm} holds in the case that $\Psi_M$ and $\Psi_N$ partition the set of ends of $T$.
\end{prop}

In \autoref{block}, we will prove that \autoref{mainthm} follows from \autoref{mainprop}. However, the heart of our proof is the proof of \autoref{mainprop}, which is the content of  \autoref{main} and \autoref{cases}.

\section{The Packing game and the Covering game}\label{games}

The purpose of this section is to define the Packing game and the Covering game as discussed
\autoref{proof_strategy} and prove some basic facts about these games.
Throughout this section, we fix a tree $T$, together with two functions $\overline M$ and $\overline N$
such that $\Tcal^M = (T,\overline M)$ and $\Tcal^N = (T,\overline N)$ are trees of matroids of overlap 1 
such that for each $t\in V(T)$ the two matroids $\overline M(t)$ and $\overline N(t)$
have the same finite ground set $E(t)$.
We denote the common underlying set of $(T,\overline M)$ and $(T,\overline N)$ by $E(T)$. We also fix some $e\in E(T)$, and $\Psi_M,\Psi_N\se \Omega(T)$.
Let $t_0$ be the unique node of $T$ with $e\in E(t_0)$.

An \emph{arena} consists of matroids $M$ and $N$ on a common finite ground set $E$, a subset $F$ of $E$ and an element $e\in E\sm F$. The set $F$ is called the set of \emph{upper edges} and $e$ is called the \emph{lower edge} of the arena.

For $t\in V(T)-t_0$, we shall later on consider the arena
$$A(t)=(M(t),N(t),E(t), F_t,  e(tt^-)),$$ where $F_t=e``(X_t)$ and
$X_t$ is the set of edges incident with $t$ and not equal to $tt^-$. 
For $t=t_0$, we take the same definition of $A(t)$ and $X_{t_0}$ except that we take the lower edge to be $e$.

The \emph{promise set} $\Pcal$ is $\{\bot, M_-,M_+,N_-,N_+, \top\}$.
Members of $\Pcal$ are called \emph{promises}.
Given an arena $(M,N,E,\emptyset,e)$, a wave $(W,S^M,S^N)$ 
in $(M,N)$ \emph{fulfils} a promise $P$ if one of the following is true:

\begin{enumerate}
\item$P=\bot$;
\item $P=M_+$ and $e$ is $M$-spanned by $W$;
  \item $P=M_-$ and $e\in S^N$;
\item $P=N_+$ and $e$ is $N$-spanned by $W$;
  \item $P=N_-$ and $e\in S^M$;
\item $P=\top$ and $e$ is both $M$-spanned and $N$-spanned by $W$;
\end{enumerate}

Note that 6 just means that $P=\top$ and $W+e$ is a hindrance focusing on $e$.

For $P,Q\in \Pcal$, we say that $P\geq_\Pcal Q$ if and only if the following is true:
Whenever there is a wave that fulfils $P$, there is also a wave that fulfils $Q$.

For example, if there is a wave $(W,S^M,S^N)$ without $e$ that $M$-spans $e$, then $(W+e, S^M,S^N+e)$ is a wave with $e$ on the $N$-side.
So $M_+\geq_\Pcal M_-$. 
Clearly, $\geq_\Pcal$ defines a partial order on $\Pcal$.

   \begin{figure} [htpb]   
\begin{center}
   	  \includegraphics[height=2cm]{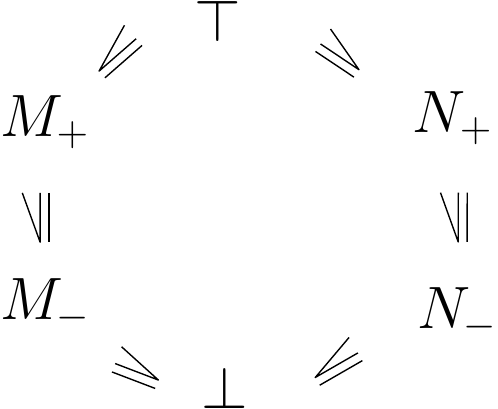}
   	  \caption{The partial order $\leq_\Pcal$.}
   	  \label{partial}
\end{center}
   \end{figure}

\begin{lem}\label{leq_P}
The partial order $\geq_\Pcal$ is the one generated from the relations
$\top\geq_\Pcal M_+, \top\geq_\Pcal N_+, M_+\geq_\Pcal M_-,N_+\geq_\Pcal N_-,M_-\geq_\Pcal \bot,N_-\geq_\Pcal \bot$, see \autoref{partial}. 
\end{lem}

\begin{proof}
It is clear that all these relations hold in $\geq_\Pcal$. 
The arenas $(U_{0,1}, U_{1,1}, \{e\}, \emptyset, e)$
and $(U_{1,1}, U_{0,1}, \{e\}, \emptyset, e)$ show that any $P\in \{M_-,M_+\}$ is incomparable with any $Q\in \{N_-,N_+\}$. \footnote{As usual, we denote by $U_{m,n}$ the uniform matroid of rank $m$ on a set of size $n$.} These arenas also show that $\top$ is strictly larger than $M_+$ and $M_-$, and that $\bot$ is strictly smaller than
$M_-$ and $N_-$.

The arena $(U_{1,2}, U_{1,2}, \{e,f\}, \emptyset, e)$ shows that $M_+$ is strictly larger than $M_-$ and also that $N_+$ is strictly larger than $N_-$. 
This shows that $\geq_\Pcal$ is generated from the relations in the lemma.
\end{proof}

We let $\Pcal^*=\{\bot^*, M_-^*,M_+^*,N_-^*,N_+^*, \top^*\}$ be the set of dual promises.
A cowave $(W,S^M,S^N)$ fulfils $P^*$ if one of 1-6 above is true with the word `spans' replaced by `cospans'.
We let $P^*\leq_{\Pcal^*} Q^*$ if and only if $P\leq_{\Pcal} Q$.

\begin{dfn}\label{relying}
Let $(M,N,E,F,e)$ be an arena and $\varphi\colon F\to \Pcal$ a function. 
Let $M'=(M/(\varphi^{-1}\{\top, M_+\})\sm (\varphi^{-1}\{\bot, N_+\})$ and
$N'=N/(\varphi^{-1}\{\top, N_+\})\sm (\varphi^{-1}\{\bot, M_+\})$.
Then a \emph{wave relying on $\varphi$} 
is a wave $(W,S^M,S^N)$ for the pair of matroids
$(M', N')$ such that $S^M \cap \varphi^{-1}(N_+) = \emptyset$ and $S^N \cap \varphi^{-1}(M_+) = \emptyset$.

Moreover, a wave relying on $\varphi$ \emph{fulfils} a promise $P$ in the arena $(M,N,E,F,e)$
if it fulfils $P$ in the arena 
$(M',N',E\sm \varphi^{-1}\{\bot, N_+,M_+, \top\}, \emptyset,e)$.
\end{dfn}

We will now explain a construction by means of which a wave for the pair $(M_{\Psi_M}(T,\overline M), M_{\Psi_N}(T,\overline N))$ can be broken down into local waves in the arenas $A(t)$ at vertices $t$ of $T$, each relying on promises fulfilled by waves higher in the tree.

\begin{constr}\label{wave_to_promise}
Let $W=(X,S^M,S^N)$ be a wave for $(M_{\Psi_M}(T,\overline M), M_{\Psi_N}(T,\overline N))$ fulfilling some promise $P$ at $e$. For each $t\in V(T)$, 
we shall construct 
a promise $P(t)$. This will induce for each $t$ a function $\varphi_t \colon F_t \to \Pcal$ sending $e(st)$ to $P(s)$. We will also construct for each $t$ a wave $W(t)$ relying on $\varphi_t$ and fulfilling $P(t)$ in the arena $A(t)$.

First we define $P(t)$. 
We let $P(t_0)=P$. If $t\neq t_0$, the construction is as follows.
We abbreviate $E_t=E(\Tcal_{t^-\to t})$. 
Very roughly, we take for $P(t)$ the strongest promise fulfilled by the wave
$(X\cap E_t, S^M\cap E_t, S^N\cap E_t)$,
possibly modified by adding $e(tt^-)$ to one of the sides of the wave. 
More precisely, 
If $Z(t)=(X\cap E_t+e(tt^-), S^M\cap E_t, S^N\cap E_t)$
is a hindrance focusing on $e(tt^-)$, we let $P(t)=\top$.
Otherwise if $Z(t)=(X\cap E_t, S^M\cap E_t, S^N\cap E_t)$
is a wave such that $S^M\cap E_t$ spans $e(tt^-)$ in $M(\Tcal_{t^-\to t}^M)$,
we let $P(t)=M_+$.
Otherwise if $Z(t)=(X\cap E_t+e(tt^-), S^M\cap E_t, S^N\cap E_t+e(tt^-))$
is a wave, we let $P(t)=M_-$.
The cases in which we take $P(t)=N_+$ or $P(t)=N_-$ are like the cases where we take
$P(t)=M_+$ or $P(t)=M_-$ but with the roles of the matroids $M$ and $N$ reversed.
In all other cases we take $P(t)=\bot$ and $Z(t)=\emptyset$.

Finally, we define $W(t)$.
Let $Z(t)=(Y(t),S^M(t), S^N(t))$ be as defined above. 
Let $F_t(M_-)$ be the set of those $e(st) \in \varphi_t^{-1}(M_-)$ such that $e(st)$ is $N$-spanned by $S^N \cap E(\Tcal_{s \to t})$ in $M_{\Psi_N}(\Tcal_{s \to t}^N)$, and let $F_t(N_-)$ be given in the same way but with the roles of $M$ and $N$ interchanged.
We let $W(t)=(Y(t)',S^M(t)', S^N(t)')$ where
$Y(t)'= Y(t)\cap E(t)\cup F_t(M_-) \cup F_t(N_-)$ and
 $S^M(t)'=S^M(t)\cap Y(t)'$
and $S^N(t)'=S^N(t)\cap Y(t)'$.

It is now straightforward to show that $W(t)$ is a wave relying on $\varphi_t$ in the arena $A(t)$ fulfilling $P(t)$.
\end{constr}

\begin{dfn}\label{tactic}
Let $A=(M,N,E,F,e)$ be an arena and let $P\in \Pcal$. Then a \emph{tactic $K$ attaining $P$ at $e$} consists of a function $\varphi_K\colon F\to \Pcal$ and a wave $(W_K, S_K^M,S_K^N)$
relying on $\varphi_K$ and fulfilling $P$, together with
sets $C_K^M$ and $C_K^N$.
If $P\in \{\top,M_+,M_-\}$, then we require that $C_K^M\in \Ccal(M)$
and that $e\in C_K^M\se S_K^M\cup \varphi^{-1}\{\top,M_+,M_-\}$.
Similarly, if  $P\in \{\top,N_+,N_-\}$, then we require that $C_K^N\in \Ccal(N)$
and that $e\in C_K^N\se S_K^N\cup \varphi^{-1}\{\top,N_+,N_-\}$.

By $\Kcal=\Kcal(A,P)$ we denote the set of all tactics $K$ attaining $P$ at $e$.
\end{dfn}

Note that  $\varphi^{-1}(M_-)\se S_K^M$, so that we could have left out $M_-$ in the term
$S_K^M\cup \varphi^{-1}\{\top,M_+,M_-\}$ above. The same remark is true for $N_-$.
A \emph{cotactic} is defined in the same way as a tactic but with a star in the appropriate places. To simplify notation, we will sometimes call cotactics just tactics.

Note that if $W$ is a wave relying on $\varphi$ and fulfilling $P$ then we can choose some sets $C^M$ and $C^N$ such that $(\varphi, W, C^M, C^N)$ is a tactic attaining $P$.

We now return to \autoref{wave_to_promise}, which started from a wave for the pair $(M_{\Psi_M}(T,\overline M), M_{\Psi_N}(T,\overline N))$ and gave us a wave in each of the arenas $A(t)$. We now show how these waves can be augmented to tactics, in a way which encodes more precisely how certain edges $e(st)$ were spanned in $M_{\Psi_M}(\Tcal^M_{t \to s})$ and $M_{\Psi_N}(\Tcal^N_{t \to s})$.

\begin{constr}\label{wave_to_tactic}
Let $(W,S^M,S^N)$ be a wave for $(M_{\Psi_M}(T,\overline M), M_{\Psi_N}(T,\overline N))$,
fulfilling some promise $P$ at $e$.
For each vertex $t$ of $T$ we will construct a tactic $K(t)=(\varphi_t, W(t),
C_{K(t)}^M, C_{K(t)}^N)$ attaining $P(t)$, with $\varphi_t$ and $W(t)$ constructed as in \autoref{wave_to_promise}.

As in \autoref{pick_circuits}, we can pick $\Psi_M$-precircuits $(S^M_t, \overline o^M_t)$ for each $t \in V(T)$ with $P(t) \in \{\top, M_+, M_-\}$, such that if
$w\in S_u^M \cap S_v^M$, then $\overline o_u^M(w)=\overline o_v^M(w)$.
Similarly, we find $\Psi_N$-precircuits
 $(S_t^N,\overline o_t^N)$ such that if
$w\in S_u^N \cap S_v^N$, then $\overline o_u^N(w)=\overline o_v^N(w)$.
If $P(t)\in \{\top, M_+,M_-\}$, we take 
$C_{K(t)}^M=\overline o^M_{t}(t)$.
If $P(t)\in \{\top, N_+,N_-\}$, we take 
$C_{K(t)}^N=\overline o^N_{t}(t)$.
We complete the definition of $K(t)$ by assigning $C_{K(t)}^M$ an arbitrary
value if 
$P(t) \notin \{\top, M_+,M_-\}$, similarly for $C_{K(t)}^N$.
\end{constr}

We have seen how to break up any wave for $(M_{\Psi_M}(T,\overline M), M_{\Psi_N}(T,\overline N))$ into tactics at each node. In \autoref{tactic_to_wave} we will show how to do the reverse: how to build a wave from a collection of local tactics, as long as they fit well together. By `fit well together' here, we mean that collectively they form a winning strategy for a particular game, which we call the Packing game.

\begin{dfn}\label{Pgame}
Let $P_0\in \Pcal$. 
The {\em Packing game} $\Gcal(P_0)  = \Gcal(T,\overline M, \Psi_M, \overline N,  \Psi_N,P_0,e)$ is played between two players, called Packer and Coverina, as follows:

Play alternates between the players, with Packer making the first move. At any point in the game there is a {\em current node} $t_c \in V(t)$, and a {\em current edge} $e_c \in E(t_c)$, and a {\em current promise $P_c\in \Pcal$}. Initially we set $e_c=e$ and $t_c = t_0$ to be the node of $T$ with $e_c \in E(t_c)$ and $P_c=P_0$. 

For any $n$ the $(2n-1)$\textsuperscript{st} move is made by Packer: he must play a 
tactic $K_n=(\varphi_{K_n}, W_{K_n}, C_{K_n}^M, C_{K_n}^N)$ that attains the promise $P_c$ in the arena $A_n=A(t_c)$.

Then the $2n$\textsuperscript{th} move is made by Coverina: she must play an edge 
$f_n\in \varphi_{t_c}^{-1}(\Pcal - \bot)$. After she does this, 
the current edge is updated to $f_n$, the current node to the unique node $t_n$ such that $f_n=e(t_{n-1}t_n)$, and  the current challenge is updated to $\varphi_{K_n}(f_n)$.

The current challenge $f_n$ is \emph{$M$-strong} if 
$\varphi_{K_n}(f_n)$ is in $\{\top, M_+,M_-\}$ and
$f_n\in C_{K_n}^M$.
Otherwise $f_n$ is $M$-\emph{weak}.
Similarly, one defines \emph{$N$-strong} and $N$-\emph{weak}.

If play continues forever, the winner is computed from the end $\omega$ of $T$ containing $(t_n | n \in \Nbb)$ and the sequences 
$(\varphi_{K_n}(f_n)| n \in \Nbb)$ and $(f_n| n \in \Nbb)$.
An end $\omega$ \emph{ is used by $M$} if 
all but finitely many $f_n$ are $M$-strong. 
Similarly, $\omega$ \emph{ is used by $N$} if 
all but finitely many $\varphi_{K_n}(f_n)$ are in 
$\{\top, N_+,N_-\}$ and $N$-strong.

Packer wins if and only if one of the following is true:
\begin{enumerate}
 \item $\omega\in \Psi_M\cap \Psi_N$;
\item $\omega\in \Psi_M$ and $\omega$ is not used by $N$;
\item $\omega\in \Psi_N$ and $\omega$ is not used by $M$;
\item $\omega$ is used by neither $M$ nor $N$;
\end{enumerate}

The {\em Covering game} $\Gcal^*(P_0) = \Gcal^*(T,\overline M, \Psi_M, \overline N,  \Psi_N,P_0,e)$ is the game like the dual Packing game $\Gcal(T,{\overline M}^*, \Psi_M\ct, {\overline N}^*,  \Psi_N\ct,P_0^*,e)$, but with the roles of Packer and Coverina reversed. We will also use a different notation for the Covering game, putting stars on the notation for the Packing game. Thus for example the current edge is denoted $e_c^*$, and Coverina's $(2n-1)$\textsuperscript{st} move is a tactic $K_n^*$, and in the $2n$\textsuperscript{th} move Packer plays some 
$f_n^*\in\{f^*\in e``F^*\mid \varphi_{K_n^*}(f^*)\neq \bot^* \}$, and 
the current challenge $f_n^*$ is \emph{$M^*$-strong} if 
$\varphi_{K_n^*}(f_n^*)$ is in $\{\top^*, M_+^*,M_-^*\}$ and
$f_n^*\in C_{K_n^*}^{M^*}$.
\end{dfn}

Given a winning strategy $\sigma$ for the Packing game, we can recover from it a subtree $Z$ of $T$ together with a tactic at each node of $Z$. We let $Z$ be the set of nodes that appear as current nodes in some play according to $\sigma$.
For each node $t$ in $Z$, there is a unique play $s_t\in \sigma$ in which $t$ is the current node - this play arises when Packer plays according to $\sigma$ and Coverina challenges on edges on the path from $t_0$ to $t$. Let $K(t)$ be the last move of Packer in $s_t$, and $P(t)$ the promise attained by $K(t)$. 

We now show how to build a wave attaining $P$ at $e$ from this collection of tactics.

\begin{constr} \label{tactic_to_wave}
Let $\tau$ be a winning strategy for Packer in the Packing game.
By modifying the tactics $K$ played by Packer according to $\tau$, we can build a winning strategy $\sigma$ with the property that if $\varphi_K(f)\in \{M_-,N_-\}$ then $e(f)\in W_K$.
Let $Z$ and the $K(t)$ and $P(t)$ be derived from $\sigma$ as above. 
Let $W=(\bigcup_{t\in Z} W_{K(t)})\cap E$, and $S^M=(\bigcup_{t\in Z} S^M_{K(t)})\cap E$, and 
$S^N=(\bigcup_{t\in Z} S^N_{K(t)})\cap E$.

First we show that $(W,S^M,S^N)$ is a wave.
Let $x\in W\sm S^M$ be arbitrary. 
Our aim is to find some $M_{\Psi_M}$-circuit $o$ such that $x\in o\se S^M+x$.
For this we need some definitions, which are illustrated in \autoref{fig:constr3}.

   \begin{figure} [htpb]   
\begin{center}
   	  \includegraphics[height=5cm]{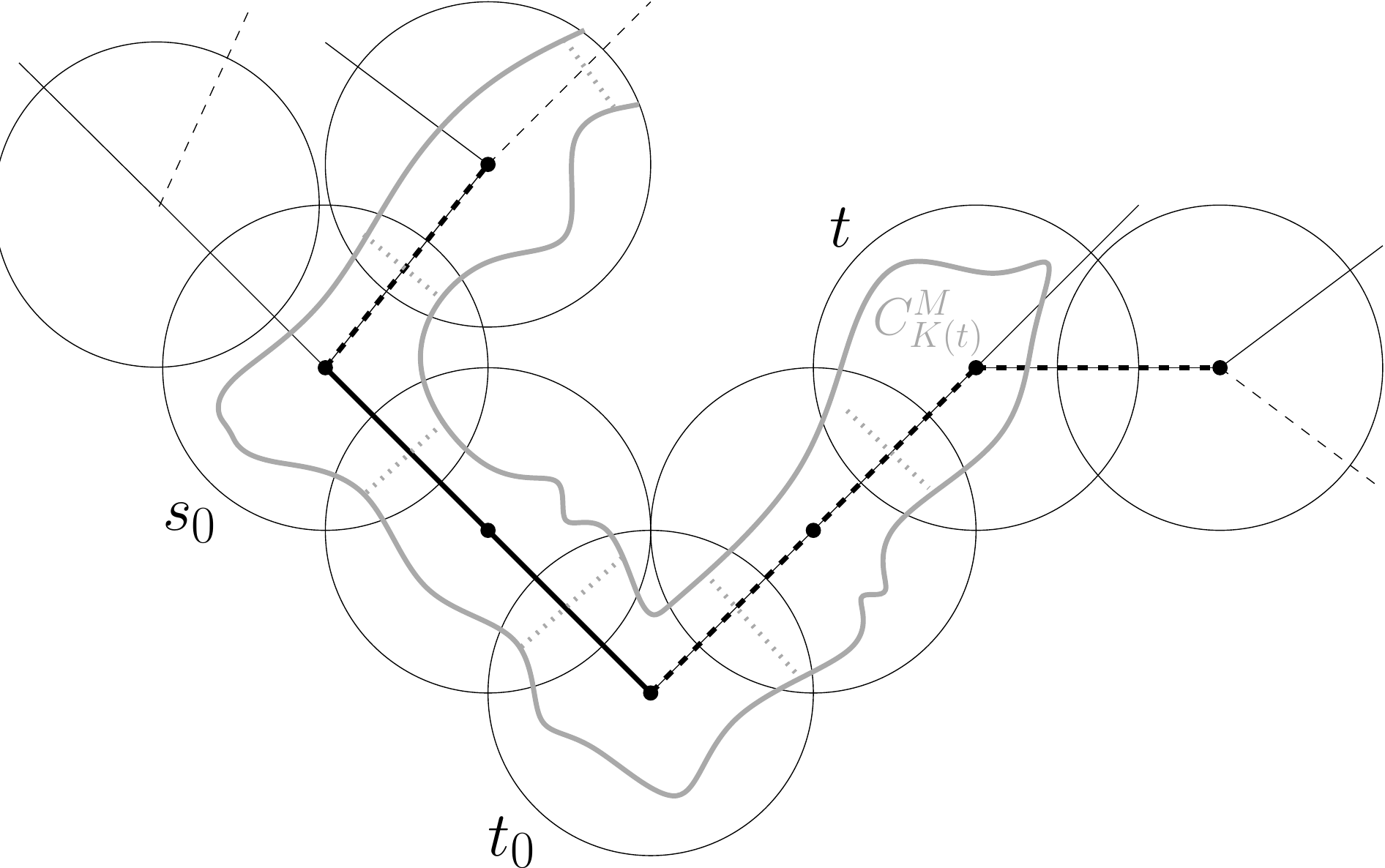}
   	  \caption{The construction of $o$. Here the highlighted path $Q$ from $s_0$ to $t_0$ has length 2 and all its 
edges are in $U_1$. The edges in $U_2$ are drawn dashed. The precircuit $(L,\overline o)$ is drawn in grey.}  	  \label{fig:constr3}
\end{center}
   \end{figure}

Let $s_0\in Z$ be such that $x\in E(M(s_0))$.
Let $Q$ be the unique path from $t_0$ to $s_0$. Note that $Q\se E(Z)$.

Let $U_1$ be the set of those edges $tu$ on $Q$ such 
that the promise fulfilled by $K(t)$
is $N_-$. 
Let $U_2\se E(Z)$ be the set of those edges $tu$ not on $Q$
such that the promise fulfilled by $K(t)$
is in $\{\top, M_+,M_-\}$.

In order to build $o$, it suffices to build a $\Psi_M$-precircuit $(L, \overline o)$ 
such that $x\in \underline{(L,\overline o)}\se S^M+x$. We shall ensure that $L\se T[U_1\cup U_2]$.

For this we first define for each $t\in T_V(U_1)\cup S_V(U_2)+s_0$ an $M(t)$-circuit $o_t\se S^M_{K(t)}\cup e''(U_1\cup U_2)+x$.
If $t=s_0$, there is such an $o_t$ with the additional property that $x\in o_t$.
Next we consider the case that $t\in T_V(U_1)$ so that there is some node $t'$ with $t't\in U_1$.
Since $\varphi_{K(t)}(e(t't))\in \{M_-,N_-\}$, the dummy edge $e(t't)$ is in $W_{K(t)}$, and thus there is such a circuit $o_t$ containing $e(t't)$.

Finally, we consider the case that $t\in S_V(U_2)$ so that there is some node $u$ with $tu\in U_2$. Here we can just take $o_t=C^M_{K(t)}$, which has the additional property that it contains $e(tu)$.

Next we define $L$.
For this we define a sequence $(L_n|n\in \Nbb)$ of sets $L_n\se V(T[U_1\cup U_2])$ with distance $n$ from $s_0$.
We start with $L_0=\{s_0\}$. 
Assume that $L_n$ is already constructed.
Let $L_{n+1}$ be the set of those nodes $w$ that have distance $n+1$ from $s_0$
such that there is some $t\in L_n$ with $e(tw)\in o_t$, where we consider $tw$ as an undirected edge.
Having defined the $L_n$, we take $L$ to be the subtree of $T[U_1\cup U_2]$ with vertex set $\bigcup_{n\in \Nbb} L_n$.
Then $(L,t\mapsto o_t)$ is a precircuit.

To see that all ends of $L$ are in $\Psi_M$, 
let $\omega$ be an end of $L$ and $R\se L$ a ray converging to $\omega$. Let $R'$ be the ray from $t_0$ which shares a tail with $R$. Let $p$ be the infinite play according to $\sigma$ obtained when Packer plays according to $\sigma$ and Coverina always challenges on edges of $R'$. Then the challenges are eventually all on edges of $U_2$, and so are $M$-strong. Thus we get an infinite play belonging to $\sigma$ which $M$-uses $\omega$.
As $\sigma$ is winning, it must be that $\omega\in \Psi_M$.
Thus $(L,t\mapsto o_t)$ is a $\Psi_M$-precircuit, giving rise to a circuit $o$, which witnesses
that $S^M$ $M_{\Psi_M}$-spans $x$ by \autoref{precirc_is_scrawl}.
Thus $S^M$ $M_{\Psi_M}$-spans $W$. Similarly one proves that $S^N$ $N_{\Psi_N}$-spans $W$.
So $(W,S^M,S^N)$ is a wave.

It remains to show that $(W,S^M,S^N)$ fulfils $P$ at $e$.
If $P\in \{\top, M_-,N_-,\bot\}$, this
follows from the fact that $(W_{K(t_0)},S^M_{K(t_0)},S^N_{K(t_0)})$ fulfils $P$ at $e$.
If $P=M_+$, then we construct an $M_{\Psi_M}$-circuit $o_e$ with $e\in o_e\se S^M+e$
in a similar way to that described above.
The case $P=N_+$ is similar. Thus $(W,S^M,S^N)$ fulfils $P$ at $e$, which completes the construction.
\end{constr}

\begin{lem}\label{Pgame_Packing}
Packer has a winning strategy $\sigma$ in the Packing game $\Gcal(P)$ if and only if there is a wave
for $(M_{\Psi_M}(T,\overline M), M_{\Psi_N}(T,\overline N))$ fulfilling $P$ at $e$. 
\end{lem}

\begin{proof}
First assume that there is a wave $(W,S^M,S^N)$
for $(M_{\Psi_M}(T,\overline M), M_{\Psi_N}(T,\overline N))$ fulfilling $P$ at $e$. 
Then Packer has the following winning strategy: at the node $v$ he plays the tactic $K(v)$ defined in \autoref{wave_to_tactic}. If Coverina challenges at some dummy edge $f$, then the new challenge is 
$\varphi_{K(v)}(f)=P_{f}$. It is straightforward to check that this is a winning strategy.

Conversely, if Packer has a winning strategy $\sigma$ then \autoref{tactic_to_wave} gives us a wave fulfilling $P$ at $e$.
\end{proof}

By duality, we get the following:

\begin{lem}\label{Cgame_Covering}
Coverina has a winning strategy $\sigma^*$ in the Covering game $\Gcal^*(P^*)$ if and only if there is a cowave
for $(M_{\Psi_M}(T,\overline M), M_{\Psi_N}(T,\overline N)$ fulfilling $P^*$ at $e$. 
\qed
\end{lem}

\begin{lem}\label{Pgame_det}
If $\Psi_M$ and $\Psi_N$ are Borel, then the Packing game is determined.
\end{lem}

\begin{proof}
Let $\Xcal$ be the set of infinite plays in the Packing game.
We endow $\Omega(T)$ with the topology inherited from the Freudenthal compactification of $T$.
For each infinite play $P$, the moves of the second player form a ray of $T$.
Let $\omega_P$ be the end this ray belongs to.
Then the function $f$ mapping $P$ to $\omega_P$ is continuous.
Thus both $f^{-1}(\Psi_M)$ and $f^{-1}(\Psi_N)$ are Borel sets.

By $S_{\{\top, M_+,M_-\}}$ we denote the set of those infinite plays whose challenges are eventually in $\{\top, M_+,M_-\}$ and $M$-strong. The set $S_{\{\top, M_+,M_-\}}$ is a countable union of closed sets and thus Borel.
Similarly, by $S_{\{\top, N_+,N_-\}}$ we denote the Borel set of those infinite plays whose challenges are eventually in $\{\top, N_+,N_-\}$ and $N$-strong. 

Now we are in a position to write the set $W$ of infinite plays in which Packer wins as a Borel set:
\[
 W=[f^{-1}(\Psi_M)\cap f^{-1}(\Psi_N)]\cup
[f^{-1}(\Psi_M)\sm S_{\{\top, N_+,N_-\}}]\cup
[f^{-1}(\Psi_N)\sm S_{\{\top, M_+,M_-\}}]\cup
\]
\[
[S_{\{\top, M_+,M_-\}} \cup S_{\{\top, N_+,N_-\}}]\ct
\]
\end{proof}

\section{Blocking sets}\label{block}

The purpose of this section is to prove that \autoref{mainprop} implies \autoref{mainthm}.
First, we turn our attention to play in an arena without upper edges: we analyse which collections of promises (or co-promises) can be fulfilled by waves (or cowaves) in such an arena. For any arena $A = (M, N, E, \emptyset, e)$, we let $\Acal(A)$ be the set of promises or co-promises fulfillable in $A$.

\begin{lem}\label{5sets}
There are precisely 5 possible values for $\Acal(A)$, as follows:
\begin{enumerate}
\item  $\Pcal + \bot^*$
\item $\Pcal^* + \bot$
\item $\{\bot, M_-, M_+, \bot^*, N_-^*, N_+^*\}$
\item $\{\bot, N_-, N_+, \bot^*, M_-^*, M_+^*\}$
\item $\{\bot, M_-, N_-, \bot^*, M_-^*, N_-^*\}$.
\end{enumerate}
\end{lem}
\begin{proof}
First we show that all 5 values are possible. We can get all but the last value from arenas with $E = \{e\}$: for the first value we take both $M$ and $N$ to be $U_{0,1}$, for the second we take both to be $U_{1,1}$, and for the third and fourth we take one to be $U_{0,1}$ and the other $U_{1,1}$. For the final value, we may take $M$ and $N$ to both be $U_{1,2}$.

Next we show that no other value is possible. We begin by showing that $\Acal(A)$ cannot contain both $M_+$ and $M_-^*$. Suppose for a contradiction that it did, and let $(W, S^M, S^N)$ be a wave fulfilling $M_+$ and $(X, T^M, T^N)$ a co-wave fulfilling $M_-^*$. By removing edges outside $W+e$ and/or contracting edges outside $X$ if necessary, we may assume without loss of generality that $W+e = X = E$. Thus $|E| \geq |S^M| + |S^N| + 1\geq r(M) + r(N \backslash e) + 1$. But also since $T^N$ is co-spanning in $N$, $T^N - e$ is co-spanning in $N \backslash e$, and so $|E| \geq |T^M| + |T^N - e| + 1 \geq r(M^*) + r((N\backslash e)^*) + 1$, so $2|E| \geq r(M) + r(M^*) + r(N \backslash e) + r((N \backslash e)^*) + 2 = 2|E| + 1$, which is the desired contradiction. 

We continue by showing that at least one of $M_+$ and $M_-^*$ must be contained in $\Acal(A)$. We begin by taking a maximal wave $(W, S^M, S^N)$ for $(M\sm e,N\sm e)$. If $e \in \Sp_M(W)$ then $W$ fulfills $M_+$. Otherwise, by contracting $W$ if necessary, we may assume that every nonempty wave contains $e$. Now we apply the Packing/Covering Theorem for finite matroids to $(M/ e, N\backslash e)$, obtaining a partition $E - e = P \dot \cup Q$ with a packing of $P$ and a covering of $Q$. Then the packing of $P$ isn't a hindrance since if it were then by \autoref{chain3} there would be a nontrivial wave for the pair $(M, N)$ not containing $e$. So it is also a covering, so that there is a cowave $(E-e, T^M, T^N)$. Now if $e$ is an $M$-loop then the empty wave fulfills $M_+$ and if not then $e$ is in the $M^*$-span of $T^M$ and so $(E, T^M, T^N + e)$ witnesses $M_-^*$.

So far we have shown that $\Acal(A)$ must contain precisely one of $M_+$ and $M_-^*$. Similarly, it must contain precisely one element of each of the sets $\{M_-, M_+^*\}$, $\{N_+, N_-^*\}$ or $\{N_-, N_+^*\}$. So if it is not given by the fifth option above, it must contain one of $M_+$, $N_+$, $M_+^*$ and $N_+^*$: without loss of generality let us say it contains $M_+$. If it also contains $N_+^*$ then, since it must be down-closed by the definition of $\leq_\Pcal$, it can only be the third option above. But if not then it must contain $N_-$. Now let $W$ be a wave fulfilling $M_+$ and let $X$ be a wave fulfilling $N_-$. Then $W \circ X$ is a wave fulfilling $\top$ and so, by down-closure again and the fact that $\bot^*$ is witnessed by the empty cowave, $\Acal(A)$ must be the first option above.
\end{proof}

\begin{rem}\label{minor_remark}
The only place where the finiteness of $E$ was used in this argument was in the application of the Packing/Covering Theorem to minors of $(M, N)$. So we get the same result without assuming finiteness of $E$, on the assumption that all minors of $(M, N)$ satisfy the Packing/Covering conjecture.
\end{rem}

\begin{dfn}
A {\em challenger} to a promise $P$ in an arena $A = (M, N, E, F, e)$ is a function $\gamma$ assigning an element of $F$ to each tactic $K$ in $\Kcal(A, P)$. For any tactic $K$ we denote by $\overline \gamma (K)$ the promise $\varphi_K(\gamma(K))$. For any $f \in F$, we denote by $\gov[f]$ the up-closure of the set $\{\gov(K) | K \in \Kcal(A, P) \text{ and } \gamma(K) = f\}$.
\end{dfn}

Challengers are important in the analysis of winning strategies in the Packing and Covering games. Let $\sigma$ be a winning strategy for Coverina in the Packing game $\Gcal(P)$, and let $s \in \sigma$ be a finite play of length $2n$.  Let $\overline s$ be $P$ if $s$ has length 0 and $\varphi_{s_{2n-1}}(s_{2n})$ otherwise (so after the play $s$ we have $P_c = \overline s$). Since $\sigma$ is winning, we may define a challenger $\gamma^{\sigma}_s$ to $\overline s$ in $A_{n+1}=A(t_c)$ by sending each tactic $K$ attaining $\overline s$ in $A_{n+1}$ to the unique $f \in e``F$ such that $s \cdot K \cdot f \in \sigma$. We omit the superscript $\sigma$ when it is clear from the context which strategy we are working with.

\begin{dfn}
A subset of $\Pcal \cup \Pcal^*$ is {\em blocking} if it meets all the possible values of $\Acal(A)$ listed in \autoref{5sets}.
\end{dfn}

That is, a set of promises and co-promises is blocking 
if for any arena with $F=\emptyset$ there exists a promise in the blocking set attainable in this arena.

\begin{lem}\label{preblock}
Let $A = (M, N, E, F, e)$ be an arena and $\rho$ a function assigning to each $f \in F$ a subset of $\Pcal \cup \Pcal^*$. Let $F'$ be the set of $f \in F$ at which $\rho(f)$ is blocking. Then there is an arena $A' = (M', N', E', F', e)$ such that for each $P \in \Pcal \cup \Pcal^*$ and tactic $K'$ attaining $P$ at $e$ in $A'$ there is a tactic $K$ attaining $P$ at $e$ in $A$ for which the function $\varphi_K$ extends $\varphi_{K'}$ and $C^M_K \cap  F' = C^M_{K'} \cap  F'$ and $C^N_K \cap F' = C^N_{K'} \cap  F'$ and for each $f \in F \setminus F'$ we have $\varphi_K(f) \not \in \rho(f)$.  
\end{lem}

\begin{rem}\label{rem_challenger}
As a consequence, for any promise $P$ and any challenger $\gamma$ to $P$ in $A$ such that $\gamma[f] \subseteq \rho(f)$ for each $f \in F$, there is a challenger $\gamma'$  to $P$ in $A'$ such that for each tactic $K$ attaining $P$ at $e$ in $A'$ there is a tactic $K$ attaining $P$ at $e$ in $A$ for which $\gamma_P(K) = \gamma_P'(K')$, the function $\varphi_K$ extends $\varphi_{K'}$, $C^M_K \cap \sm F' = C^M_{K'} \cap F'$ and $C^N_K \cap F' = C^N_{K'} \cap F'$.  
\end{rem}

\begin{proof}[Proof of \autoref{preblock}]
For each $f \in F \setminus F'$, choose one of the 5 sets from \autoref{5sets} which $\rho(f)$ fails to meet, and let $F_i$ be the set of those $f \in F \setminus F'$ for which the $i$\textsuperscript{th} element of the list was chosen. Let $E' = E \setminus (F_1 \cup F_2 \cup F_3 \cup F_4)$, $M' = M /(F_1 \cup F_3) \backslash (F_2 \cup F_4)$ and $N' =  N/(F_1 \cup F_4) \backslash (F_2 \cup F_3)$. As in the statement, let $A' = (M', N', E', F', e)$. 

Let $K'$ be a tactic attaining some promise $P$ at $e$ in $A'$. We define the corresponding tactic $K$ in $A$ as follows: let $W_K = W'_{K'}$, and let $\varphi_{K_n}$ be obtained as the extension of $\varphi_{K_n}'$ to $F$ taking the value $\top$ on $F_1$, $\bot$ on $F_2$, $M_+$ on 
$F_3$, $N_+$ on $F_4$, $M_-$ on $F_5 \cap S^M$ and $N_-$ on $F_5 \setminus S^M$. Let $C^M_K$ be an extension of $C^M_{K'}$ whose new elements all come from $F_1 \cup F_3$ and $C^N_K$ be an extension of $C^N_{K'}$ whose new elements all come from $F_1 \cup F_4$. 
\end{proof}

\begin{lem}\label{blocklem}
Let $B$ be a blocking set and $A = (M, N, E, F, e)$ an arena. For each $P \in B$, let $\gamma_P$ be a challenger to $P$ in $A$. Then there is some $f \in F$ such that $\bigcup_{P \in B} \gov_P[f]$ is blocking.
\end{lem}
\begin{proof}
Suppose for a contradiction that there is no such $f$. Then we apply \autoref{preblock}
with $\rho \colon f \mapsto \bigcup_{P \in B} \gov_P[f]$ and get an arena with no upper edges 
in which none of the promises in $B$ can be attained by any tactic, contradicting the fact that $B$ is blocking.
\end{proof}

This useful property makes it worth looking at blocking sets in detail, and we will now pause to analyse their structure more carefully. However, we shall only consider up-closed blocking sets: note that a set is blocking if and only if its up-closure is and 
by the definitions of challenger and of $\leq_\Pcal$, if we have a challenger to every element of some blocking set then we also get a challenger to every element of its up-closure.

\begin{lem}\label{blockstr}
An up-closed set is blocking if and only if it includes one of the following sets as a subset: $\{\bot\}$, $\{\bot^*\}$, $\{M_+, M_-^*\}$, $\{M_-, M_+^*\}$, $\{N_+, N_-^*\}$, $\{N_-, N_+^*\}$, $\{M_+, N_-, \top^*\}$, $\{M_-, N_+, \top^*\}$, $\{M_+^*, N_-^*, \top\}$ or $\{M_-^*, N_+^*, \top\}$.
\end{lem}
\begin{proof}
Each of the listed sets is clearly blocking. Conversely, let $B$ be an up-closed blocking set. Since it meets $\{\bot, M_-, N_-, \bot^*, M^*_-, N^*_-\}$ and is up-closed it must contain one of $M_-$, $N_-$, $M_-^*$ or $N_-^*$: by symmetry we may assume without loss of generality that it contains $M_-$. Since it meets $\Pcal^* \cup \bot$ it must contain one of $\top^*$ and $\bot$ and since it meets $\{\bot, N_-, N_+, \bot^*, M_-^*, M_+^*\}$ it must contain one of $N_+$ and $M_+^*$. Now if it contains $\bot$ then it includes $\{\bot\}$, if it contains $M_+^*$ then it includes $\{M_-, M_+^*\}$, and if it contains neither then it contains both of $\top^*$ and $N_+$ and so includes $\{M_-, N_+, \top^*\}$.
\end{proof}

We are now ready to prove the main result of this section:

\begin{proof}[Proof that \autoref{mainprop} implies \autoref{mainthm}]
Suppose that we have  two trees of matroids as in the statement of \autoref{mainthm}. By  \autoref{PCfromwave} and \autoref{ToMminor}, it suffices to prove that every edge either lies in some wave or else is the focus of some cohindrance. So let $e$ be some edge. We now consider the Packing and Covering games $\Gcal(M_-)$, $\Gcal(N_-)$ and $\Gcal^*(\top^*)$, taking our notation as in \autoref{Pgame}. If Packer has a winning strategy in either of $\Gcal(M_-)$ or $\Gcal(N_-)$ or Coverina has a winning strategy in $\Gcal^*(\top^*)$ then we are done by \autoref{Pgame_Packing} or \autoref{Cgame_Covering}. So we suppose for a contradiction that there are no such strategies.

By the determinacy of these games (\autoref{Pgame_det}) we get winning strategies $\sigma_{M_-}$ and $\sigma_{N_-}$ for Coverina in $\Gcal(M_-)$ and $\Gcal(N_-)$ and a winning strategy $\sigma_{\top^*}$ for Packer in $\Gcal^*(\top^*)$. Let $\sigma$ be the union of these three strategies. For any finite play $s$, let $l(s)$ be the last move of $s$. Note that if $s \in \sigma$ then $l(s)$ is always of the form $e(f)$ for some $f \in E(T)$. For any edge $tt'$ of $T$, let $\sigma[tt']$ be $\{\overline s | s \in \sigma \text{ and } l(s) = e(tt')\}$. Let $U$ be the set of edges $tt'$ of $T$ at which $\sigma[tt']$ is blocking, and let $T'$ be the subtree of $T$ on the vertices which can be joined to $t_0$ via a path all of whose edges are in $U$.

We now define two trees of matroids on $T'$, to which we will apply \autoref{mainprop} to obtain the desired contradiction. For each $u \in T'$, we apply \autoref{preblock} to the arena $A(u)$ and the function $\rho \colon e(tu) \mapsto \sigma[tu]$ to get a new arena $A'(u) = (\overline M'(u), \overline N'(u), E'(u), F'_u, e(uu^-)$, where we choose the underlying sets $E'(u)$in such a way that all the sets $E'(u) \sm(F_u + e(uu^-))$ are disjoint and contain no dummy edges. Then $(T', \overline M')$ and $(T', \overline N')$ are trees of matroids.

Now we consider the Packing game $\Gcal'(M_-) = \Gcal(T', \overline M ', \Psi_M \cap \Omega(T'), \overline N ', \Psi_N \cap \Omega(T'), M_-, e)$. We will use a slightly different notation for this game than for $\Gcal(M_-)$, putting dashes on the notation used in $\Gcal(M_-)$. Thus, for example, the current promise at any point is denoted by $P_c'$. We can convert $\sigma_{M_-}$ into a winning strategy for Coverina in $\Gcal'(M_-)$ as follows: Coverina should imagine an auxiliary play in the game $\Gcal(M_-)$, in which she plays according to $\sigma_{M_-}$, and for which she should ensure that at any point the current node, edge and promise agree with those in $\Gcal'(M_-)$. When Packer plays a tactic $K_n'$, Coverina should choose a tactic $K_n$ attaining $P_c = P_c'$ in $A_c$ as in \autoref{preblock} and she should let her response $f_n'$ be the move $f_n$ prescribed by $\sigma_{M_-}$ in response to $K_n$.

By \autoref{Pgame_Packing}, the existence of this winning strategy entails that there is no wave fulfilling $M_-$ in $(M_{\Psi_M \cap \Omega(T')}(T', \overline M'), (M_{\Psi_N \cap \Omega(T')}(T', \overline N'))$. Similarly, there is no wave fulfilling $N_-$ and no cowave fulfilling $\top^*$ for this pair. By \autoref{minor_remark}, since the set $\{M_-, N_-, \top^*\}$ is blocking, there is some minor of this pair for which Packing/Covering fails to hold. Thus in order to obtain the desired contradiction by applying \autoref{mainprop} to this minor, we just need to show that every end of $T'$ is in $\Psi_M \triangle \Psi_N$.

Let $\omega = (t_n | n \in \Nbb)$ be an end of $T'$. For each $n$ the set $\sigma[t_nt_{n+1}]$ is blocking and does not contain $\bot^*$, so must meet $\Pcal$. Thus there must be some play $s \in \sigma_{M_-} \cup \sigma_{N_-}$ with $l(s) = e(t_nt_{n+1})$. Since there are only finitely many such plays for each $n$, we obtain by \autoref{Infinity_Lemma} that there is some infinite play $\hat s$ according to one of $\sigma_{M_-}$ or $\sigma_{N_-}$ with $\hat s_{2n} = e(t_nt_{n+1})$ for each $n$. But then since these strategies are winning for Coverina, it follows that $\omega$ must be in at most one of $\Psi_M$ and $\Psi_N$. A similar argument shows that it is also in at least one of $\Psi_M$ and $\Psi_N$, so that it is in $\Psi_M \triangle \Psi_N$ as required.

\end{proof}

\section{Main result}\label{main}

As we have just shown, in order to prove our main result it remains to prove the special case given in \autoref{mainprop}.

Throughout this section we fix two trees of matroids as in the statement of \autoref{mainprop}. Our aim is to show that the pair $(M_{\Psi_M}(T, M), M_{\Psi_N}(T, N))$ satisfies matroid intersection. We shall suppose that it does not, and in the remainder of this section we will derive a contradiction from that supposition. However, it will become clear during the course of the proof that we must rely on two technical lemmas, whose proofs we defer to the next section.

By \autoref{ToMminor} and \autoref{PCfromwave}, we may assume that there is some edge $e$ of $E(T)$ which is not in any wave or cowave for our pair of matroids. We now consider the Packing and Covering games $\Gcal(M_-)$ and $\Gcal^*(M_+^*)$, taking our notation as in \autoref{Pgame}. By \autoref{Pgame_Packing} Packer does not have a winning strategy in $\Gcal(M_-)$, and by \autoref{Cgame_Covering} Coverina does not have a winning strategy in $\Gcal^*(M_+^*)$. So by the determinacy of these games, there are winning strategies $\sigma_{M_-}$ for Coverina in $\Gcal(M_-)$ and  $\sigma_{M^*_+}$ for Packer in $\Gcal^*(M_+^*)$. Let $\sigma$ be the union of these strategies. 

Let $t_0$ be the unique vertex of $T$ with $e \in E(t_0)$.
In order to get a contradiction, we shall recursively construct two infinite plays $s_{M_-}$
and $s_{M_+^*}$ in $\Gcal(M_-)$ and $\Gcal^*(M_+^*)$ respectively.
We shall construct $s_{M_-}$
and $s_{M_+^*}$ such that they are both played along the same ray $(t_i|i\in \Nbb)$ from $t_0$
and such that either Packer wins $s_{M_-}$ or Coverina wins $s_{M_+^*}$.

More explicitly, let us say that a finite or infinite play $s$ is $(M, i)$-weak if there is some $j \geq i$ such that the challenge $s_{2j}$ is defined and $M$-weak. We define $(N, i)$-weak, $(M^*, i)$-weak and $(N^*, i)$-weak similarly. We shall recursively build a ray $(t_i | i \in \Nbb)$ from $t_0$ in $T$ and sequences $(B_i| i \in \Nbb)$ of blocking sets and $(\lambda_i \colon B_i \to \sigma | i \in \Nbb)$ of functions, with the following properties:
\begin{enumerate}
\item $B_0 = \{M_-,M_+^*\}$ and $\lambda_0$ sends both elements to trivial plays.
\item Each of the sets $B_i$ is one of the blocking sets listed in \autoref{blockstr}.
\item For any $P \in B_i$ with $i > 0$ the play $\lambda_i(P)$ is a play in $\sigma$ with last move $e(t_{i-1}t_i)$ and $\overline {\lambda_{i}(P)} \leq P$.
\item For any $P, Q \in B_{i}$ with $\overline{\lambda_i(P)} = \overline{\lambda_i(Q)}$ we have $\lambda_i(P) = \lambda_i(Q)$.
\item For any $i > 0$ and any $P \in B_i$ there is some $P' \in B_{i-1}$ such that $\lambda_i(P)$ is an extension of $\lambda_{i-1}(P')$.
\item For any $i$ there is some $j \geq i$ such that one of the following is true:
\begin{itemize}
\item For each $P \in B_j \cap \Pcal$ the play $\lambda_j(P)$ is $(N, i)$-weak.
\item For each $P \in B_j \cap \Pcal^*$ the play $\lambda_j(P)$ is $(M^*, i)$-weak.
\end{itemize}
\item For any $i$ there is some $j \geq i$ such that one of the following is true:
\begin{itemize}
\item For each $P \in B_j \cap \Pcal$ the play $\lambda_j(P)$ is $(M, i)$-weak.
\item For each $P \in B_j \cap \Pcal^*$ the play $\lambda_j(P)$ is $(N^*, i)$-weak.
\end{itemize}
\end{enumerate}

It is possible to recursively build a sequence satisfying 1-5  by \autoref{blocklem}. To get the additional conditions, we will need to make use of the results of \autoref{cases}. But before we do this, we will explain why the existence of such a sequence would result in a contradiction. As each end of $T$ is in precisely one of $\Psi_M$ and $\Psi_N$, we may without loss of generality suppose that the end $(t_i | i \in \Nbb)$ of $T$ is in $\Psi_N \setminus \Psi_M$. Since each $B_i$ is finite, by \autoref{Infinity_Lemma}, we can find an infinite play $s_{M_-}$ such that for each $i \in \Nbb$ the restriction $s_{M_-} \restric_{2i}$ is in both $\sigma_{M_-}$ and the image of $\lambda_i$. Thus $s_{M_-}$ is an infinite play according to $\sigma_{M_-}$. Since this is a winning strategy for Coverina, there must be some $i_{M_-}$ such that $(s_{M_-})_{2j}$ is never an $N$-weak challenge for $j \geq i_{M_-}$. Similarly, we can build an infinite play $s_{M_+^*}$ such that for each $i \in \Nbb$ the restriction $s_{M_+^*
} \restric_{2i}$ 
is in both $\sigma_{M^*_+}$ and the image of $\lambda_i$, and there is some $i_{M_+^*}$ such that $(s_{M_+^*})_{2i}$ is never an $M^*$-weak challenge for $j \geq i_{M_+^*}$. Now let $i$ be whichever of $i_{M_-}$ and $i_{M_+^*}$ is larger, and apply condition 6 above. If the first option holds, then $s_{M_-}$ is $(N, i)$-weak, contrary to the construction of $i$. But if the second option holds then $s_{M_+^*}$ is $(M^*, i)$-weak, which is again a contradiction.

So to complete our proof it remains to show how we can ensure that the sequence we recursively construct satisfies the 6\textsuperscript{th} and 7\textsuperscript{th} conditions above. In order to do this, it is enough to show how, given choices of $t_k$, $B_k$ and $\lambda_k$ for $k \leq i$ satisfying 1-5 we can extend these finite sequences to longer finite sequences $(t_k | k \leq j)$, $(B_k | k \leq  j)$ and $(\lambda_k | k \leq j)$ for some $j \geq i$ such that one of the following is true:
\begin{itemize}
\item For each $P \in B_j \cap \Pcal$ the play $\lambda_j(P)$ is $(N, i)$-weak.
\item For each $P \in B_j \cap \Pcal^*$ the play $\lambda_j(P)$ is $(M^*, i)$-weak.
\end{itemize}

For if we can do this, then we can use a symmetrical construction to further extend our sequences to $(B_k | k \leq j')$ and $(\lambda_k | k \leq j')$ for some $j' \geq j$ such that one of the following is true:
\begin{itemize}
\item For each $P \in B_j \cap \Pcal$ the play $\lambda_j(P)$ is $(M, i)$-weak.
\item For each $P \in B_j \cap \Pcal^*$ the play $\lambda_j(P)$ is $(N^*, i)$-weak.
\end{itemize}
Repeatedly carrying out this pair of constructions and, if they don't make the sequences longer, extending them using \autoref{blocklem}, we will obtain infinite sequences satisfying all the conditions above.

So suppose that we are given choices of $t_k$, $B_k$ and $\lambda_k$ for $k \leq i$ and that we wish to extend these sequences to satisfy condition 6 at $i$. The way we do this depends on the value of $B_i$. We cannot, by the construction of the Packing and Covering games, have $B_i = \{\bot\}$ or $B_i = \{\bot^*\}$. If $B_i = \{M_+, M_-^*\}$, then  we are done, since the play $\lambda_i(M_+)$ is necessarily $N$-weak. The cases where $B_i$ is one of $\{M_-, M_+^*\}$, $\{N_+, N_-^*\}$ or $\{N_-, N_+^*\}$ are dealt with similarly.

The next case, $B = \{M_-, N_+, \top^*\}$, is a little trickier. Here we may be forced to extend the sequence. The object we need in order to do this is encoded in the following definition:

\begin{dfn}\label{tactician+}
 Let $A = (M, N, E, F, e)$, and $ \gamma_{M_-}$, $ \gamma_{N_+}$ and $ \gamma_{\top^*}$ be challengers to the respective promises.
A {\em tactician$^+$} for a blocking set $B$ at an edge $f\in F$ in $A$ is a function $\mu$ sending each $P$ in $B$ to a pair $(Q, K)$, where $Q\in \{M_-,N_+,\top^*\}$ and $K$ is a tactic attaining $Q$ in $A$ and $\varphi_K(f) \leq P$
and $\gamma_Q(K)=f$.
\end{dfn}

Note that in the context of \autoref{tactician+}, $F$ cannot be empty since $\{M_-,N_+,\top^*\}$ is blocking. 

We are interested in the case where $\gamma_{M_-} = \gamma^{\sigma_{M_-}}_{\lambda_i(M_-)}$ (the challenger determined by the strategy $\sigma_{M_-}$ after the finite play $\lambda_i(M_-)$), $\gamma_{N_+} = \gamma^{\sigma_{M_-}}_{\lambda_i(N_+)}$ and $\gamma_{\top^*} = \gamma^{\sigma_{M_+^*}}_{\lambda_i(\top^*)}$
In this context, given such $f$, $B$ and $\mu$, we can extend our sequences as follows: we choose $t_{i+1}$ with $f = e(t_it_{i+1})$, we choose $B_{i+1}$ to be $B$, and for each $P \in B_{i+1}$ we take $\lambda_{i+1}(P)$ to be the play consisting of $\lambda_i(\pi_1(\mu(P)))$ followed by the tactic $\pi_2(\mu(P))$ and then the edge $f$. We must be able to find some extension like this by the following lemma:

\begin{lem}\label{tactician+_exists}
For each  $f\in F$ and blocking set $B$ included 
in $\overline\gamma_{M_-}[f]\cup \overline\gamma_{N_+}[f]\cup \overline\gamma_{\top^*}[f]$, there is a 
tactician$^+$ $\mu_B$ for $B$ at $f$.
\end{lem}

\begin{proof}
 For each $P\in B$, the promise $P$ is in $\overline \gamma_{Q}[f]$ for some $Q\in \{M_-,N_+,\top^*\}$. 
Then there is a tactic $K$ fulfilling $Q$ at $e$ such that $\varphi_K(f)=P$ and 
$\gamma_{M_-}(K)=f$. We let $\mu_B(P)=(Q,K)$.
\end{proof}

In order to ensure that our extension is helpful, we use the following lemma, to be proved in the next section:

\begin{lem}\label{tactician+_lem}
Let $A = (M, N, E, F, e)$ be an arena,  $ \gamma_{M_-}$, $ \gamma_{N_+}$ and let $ \gamma_{\top^*}$ be challengers as in \autoref{tactician+}. Then there are a blocking set $B$, an edge $f \in F$ and a tactician$^+$ $\mu$ for $B$ at $f$ in $A$ such that one of the following holds:
\begin{enumerate}[(i)]
\item Double Extension case: $B= \{M_-,N_+,\top^*\}$ and $\pi_1(\mu(P))=P$ for each $P\in B$;
\item Weak Challenge case in the Packing game:
For any tactic $K$ with $(N_+,K)$ in the image of $\mu$, the edge $f$ is an $N$-weak challenge to $K$;
\item Weak Challenge case in the Covering game:
For any tactic $K$ with $(\top^*,K)$ in the image of $\mu$, the edge $f$ is an $M^*$-weak challenge to $K$.
\end{enumerate}
\end{lem}

The Weak Challenge cases are self-explanatory: for example, if we have the Weak Challenge case in the Packing game then this ensures that for each $P \in B_{i+1} \cap \Pcal$ the play $\lambda_{i+1}(P)$ is $(N, i)$-weak. The Double Extension case is more subtle. In this case, we find ourselves in the same situation we were in before, with $B_{i+1} = \{M_-, N_+, \top^*\}$, but we don't seem to have made any progress. However, we can apply the Lemma again repeatedly to get a contradiction as follows: 

Suppose for a contradiction that there is no finite $j \geq i$ for which there are extensions $(B_k | k \leq j)$ and $(\lambda_k | k \leq j)$ of our sequences which satisfy condition 6 at $i$ and $j$. Then we recursively build sequences $(B_j | j > i)$, $(\lambda_j | j > i)$ and $(t_j | j > i)$, where for each $j > i$ we choose $t_j$, $B_j$ and $\lambda_j$ as above using \autoref{tactician+_lem}. Since as we have noted by our supposition we never have either challenge case, we get that $B_j = \{M_-, N_+, \top^*\}$,  $\lambda_j(M_-)$ extends $\lambda_{j-1}(M_-)$ and $\lambda_j(N_+)$ extends $\lambda_{j-1}(N_+)$ for each $j > i$. So there are two infinite plays $u_{M_-}$ and $u_{N_+}$ according to $\sigma_{M_-}$ such that, for each $j \geq i$, each $u_P$ extends $\lambda_j(P)$. Let $\omega$ be the end $(t_k|k \in \Nbb)$. As $\sigma_{M_-}$ is winning and in $u_{M_-}$ all challenges are eventually $N$-weak, we must have $\omega  \not \in \Psi_M$. Similarly $\omega \not \in \Psi_N$, which is the desired 
contradiction.

Thus there is some finite $j \geq i$ for which there are extensions $(B_k | k \leq j)$ and $(\lambda_k | k \leq j)$ of our sequences which satisfy condition 6 at $i$ and $j$, as required. This completes our treatment of the case $B = \{M_-, N_+, \top^*\}$. The case $B = \{ M_+^*, N_-^*,\top\}$ is similar, using the dual of \autoref{tactician+_lem}.

The case $B = \{M_+, N_-, \top^*\}$, is very similar, but there is an additional complexity. Once more we may be forced to extend the sequences $(t_k)$, $(B_k)$ and $(\lambda_k)$. This time the object we need in order to do this is encoded in the following definition:

\begin{dfn}\label{tactician-}
 Let $A = (M, N, E, F, e)$, and $ \gamma_{M_+}$, $ \gamma_{N_-}$ and $ \gamma_{\top^*}$ be challengers to the respective promises.
A {\em tactician$^-$} for a blocking set $B$ at an edge $f\in F$ in $A$ is a function $\mu$ sending each $P$ in $B$ to a pair $(Q, K)$, where $Q\in \{M_+,N_-,\top^*\}$ and $K$ is a tactic attaining $Q$ in $A$ and $\varphi_K(f) \leq P$
and $\gamma_Q(K)=f$.
\end{dfn}

Note that in the context of \autoref{tactician-}, $F$ cannot be empty since $\{M_+,N_-,\top^*\}$ is blocking.

We are interested in the case where $\gamma_{M_+} = \gamma^{\sigma_{M_-}}_{\lambda_i(M_+)}$, $\gamma_{N_-} = \gamma^{\sigma_{M_-}}_{\lambda_i(N_-)}$ and $\gamma_{\top^*} = \gamma^{\sigma_{M_+^*}}_{\lambda_i(\top^*)}$. 
In this context, given such $f$, $B$ and $\mu$, we can extend our sequences as follows: we choose $t_{i+1}$ with $f = e(t_it_{i+1})$, we choose $B_{i+1}$ to be $B$, and for each $P \in B_{i+1}$ we take $\lambda_{i+1}(P)$ to be the play consisting of $\lambda_i(\pi_1(\mu(P)))$ followed by the tactic $\pi_2(\mu(P))$ and then the edge $f$. We must be able to find some extension like this by the following lemma, which can be proved similarly to \autoref{tactician+_exists}:

\begin{lem}\label{tactician-_exists}
For each  $f\in F$ and blocking set $B$ included 
in $\overline\gamma_{M_+}[f]\cup \overline\gamma_{N_-}[f]\cup \overline\gamma_{\top^*}[f]$, there is a 
tactician$^-$ $\mu_B$ for $B$ at $f$.\qed
\end{lem}

In order to ensure that our extension is helpful, we will once more rely on a technical lemma, to be proved in the next section:

\begin{lem}\label{tactician-_lem}
Let $A = (M, N, E, F, e)$ be an arena,  $ \gamma_{M_+}$, $ \gamma_{N_-}$ and let $ \gamma_{\top^*}$ be challengers as in \autoref{tactician-}. Then there are a blocking set $B$, an edge $f \in F$ and a tactician$^-$ $\mu$ for $B$ at $f$ in $A$ such that one of the following holds:
\begin{enumerate}[(i)]
\item Double Extension case: $B= \{M_+,N_-,\top^*\}$ and $\pi_1(\mu(P))=P$ for each $P\in B$;
\item Weak Challenge case in the Packing game:
For any tactic $K$ with $(N_-,K)$ in the image of $\mu$, the edge $f$ is an $N$-weak challenge to $K$;
\item Weak Challenge case in the Covering game:
For any tactic $K$ with $(\top^*,K)$ in the image of $\mu$, the edge $f$ is an $M^*$-weak challenge to $K$;
\item Improvement case 1: 
$B=\{M_-,N_+,\top^*\}$;
\item Improvement case 2: 
$B=\{N_-^*,M_+^*,\top\}$.
\end{enumerate}
\end{lem}

The Weak Challenge cases are once more self-explanatory, and the Double Extension case can be dealt with as before. But Improvement cases 1 and 2 reduce the situation to one in which the current blocking set is $\{M_-, N_+, \top^*\}$ or $\{ M_+^*, N_-^*,\top\}$, and both of these situations have been dealt with above.

This completes our treatment of the case $B = \{M_+, N_-, \top^*\}$. The case $B = \{M_-^*, N_+^*,\top\}$ is similar, using the dual of \autoref{tactician+_lem}. We have now dealt with all cases which can arise, and this completes the proof of \autoref{mainprop} and hence of \autoref{mainthm}.

\section{Proof of Lemmas \ref{tactician+_lem} and \ref{tactician-_lem}}\label{cases}

\subsection{Proof of \autoref{tactician+_lem}}

The aim of this subsection is to prove \autoref{tactician+_lem}.
First we need some intermediate lemmas.
We start with a lemma on waves in finite pairs of matroids.

\begin{lem}\label{lemma27}
 Let $M$ and $N$ be two matroids on the same finite ground set $E$.
Let $G,H,J\se E$ disjoint  and $e\in E\sm (G\cup H\cup J)$.
Then one of the following is true.
\begin{enumerate}
\item There is a wave with $e$ on the $N$-side in $(M/(H\cup J), N/(H\cup J))$.
\item There is a wave $N$-spanning $e$ in $(M\sm (G\cup J),N\sm G/J)$.
\item There is some $G'\se G$ and a cohindrance $(Y,T^M,T^N)$ focusing on $e$ in 
$(M\sm (G'\cup J),N\sm G'/ J)$
such that there is some $M$-cocircuit $b$ with $e\in b\se (T^M+e)\sm H$. 
\end{enumerate}

\end{lem}

If $G=H=J=\emptyset$, then this lemma just says that $\{M_-,N_+,\top^*\}$ is blocking.
So this lemma can be seen as an extension of this fact.

\begin{proof}

We assume that we do not have outcome 1 or 2 and aim to show that then we get outcome 3.
Thus as $\{M_-,N_+,\top^*\}$ is blocking by \autoref{blockstr}, 
in the pair $(M',N')=(M\sm (G\cup J),N\sm G/J)$
it must be that the promise $\top^*$ is attainable: There is a cohindrance $(X,S^M,S^N)$ focusing on $e$.

Now we tweak this cohindrance a little to get outcome 3.
As $\{M_-,M_+^*\}$ is blocking by \autoref{blockstr} and we do not have outcome 1, 
in the pair $(M/(H\cup J),N/(H\cup J))$ there is a cowave $(Y, T^M,T^N)$ that $M$-cospans $e$. 
In particular, there is an $M$-cocircuit $b$ with $e\in b\se (T^M+e)$. So $b$ avoids $H\cup J\cup G'$, where $G'=G\sm Y$.
Then $(X\sm Y, S^M\sm Y, S^N\sm Y)$ is a cohindrance focusing on $e$ in the pair $(M'\sm Y,N'\sm Y)$.
By the dual of \autoref{joinwaves} $(X\cup Y, (S^M\sm Y)\cup T^M, (S^N\sm Y)\cup T^N)$ is a 
cohindrance in $(M\sm (G'\cup J), N\sm G'/J)$, and together with $b$ it witnesses that we have  outcome 3.
\end{proof}

\autoref{lemma27} is the main principle we use in the proof of \autoref{tactician+_lem}.
The work of bridging from \autoref{lemma27} to \autoref{tactician+_lem} is done in the following lemma.

\begin{lem}\label{lem5_minus}
Let $\overline \gamma_{M_-}$, $\overline\gamma_{N_+}\se \Pcal-\bot$ and $\overline\gamma_{\top^*}\se \Pcal^*-\bot^*$
be up-closed such that $\overline \gamma=\overline\gamma_{M_-}\cup\overline\gamma_{N_+}\cup\overline\gamma_{\top^*}$ is blocking.
Then one of the following is true. 
\begin{enumerate}
 \item One of the 4 sets $\{M_+,M_-^*\}$, $\{M_-,M_+^*\}$, $\{N_+,N_-^*\}$ or $\{N_-,N_+^*\}$
is a subset of $\overline \gamma$;
\item $M_- \in \overline \gamma_{M_-}$ and $\{N_+, \top^*\} \se \overline \gamma$;
\item $N_-\in \overline\gamma_{M_-}$ and 
$\{M_+,\top^*\}\se \overline\gamma$;
\item $\top\in \overline\gamma_{M_-}$ and 
one of $\{M_-^*,N_+^*\}\se \overline\gamma_{\top^*}$ or $\{M_+^*,N_-^*\}\se \overline\gamma_{\top^*}$;
\item $\overline\gamma_{M_-}\se \{M_+,N_+,\top\}$ and 
$\overline\gamma_{\top^*}\se \{M_+^*,N_+^*, \top^*\}$ and 
one of $\{M_-,N_+\}\se \overline\gamma$ or
$\{M_+,N_-\}\se \overline\gamma$;

\item $\overline\gamma_{M_-}=\emptyset$ and
$N_-^*\not\in \overline\gamma_{\top^*}$ and $\{M_-^*,N_+^*,\top\}\se\overline\gamma$
and $\overline\gamma_{N_+}\se \{N_+,\top\}$;

\item $\overline\gamma_{M_-}=\emptyset$ and
$\{M_+^*,N_-^*,\top\}\se\overline\gamma$
and $\overline\gamma_{N_+}\se \{M_+,\top\}$;

\end{enumerate}
\end{lem}

\begin{proof}
Since $\overline\gamma\cap(\Pcal+\bot^*)$ is nonempty and $\bot^*\notin \overline\gamma$, 
we get that $\overline\gamma\cap\Pcal$ is nonempty, thus $\top\in \overline\gamma$. 
Similarly, $\top^*\in \overline\gamma$.
Now suppose for a contradiction that we do not have one of the outcomes 1-7.
By \autoref{blockstr}, 
one of the 4 sets  $\{M_+,N_-,\top^*\}$, $\{M_+^*,N_-^*,\top\}$, $\{M_-,N_+,\top^*\}$ or
$\{M_-^*,N_+^*,\top\}$
is a subset of $\overline \gamma$.

\paragraph{Case 1: $\{M_+,N_-,\top^*\}\se \overline \gamma$ or $\{M_-,N_+,\top^*\}\se \overline \gamma$.}
Then $M_+$ and $N_+$ are in $\overline \gamma$, so $M_-^*$ and $N_-^*$ are not as we do not have outcome 1.
Also, $M_- \not \in \overline \gamma_{M_-}$ as we do not have outcome 2, and $N_- \not \in \overline \gamma_{M_-}$ as we do not have outcome 3. 
Thus we have outcome 5, which is the desired contradiction.
\paragraph{Case 2: $\{M_+^*, N_-^*, \top\} \se \overline \gamma$.}
Then $M_-$ and $N_+$ cannot be in $\overline \gamma_{N_+}$ as we do not have outcome 1. Also, $\overline \gamma_{M_-}$ must be empty as we do not have outcome 4. Thus we have outcome 7, which is the desired contradiction.
\paragraph{Case 3: $\{M_-^*, N_+^*, \top\} \se \overline \gamma$ but $\{M_+^*, N_-^*, \top\} \not \se \overline \gamma$.}  
Then $M_+$ and $N_-$ cannot be in $\overline \gamma_{N_+}$ as we do not have outcome 1. By assumption, $N_-^* \not \in \overline \gamma_{\top^*}$. Also, $\overline \gamma_{M_-}$ must be empty as we do not have outcome 4. Thus we have outcome 6, which is the desired contradiction.
\end{proof}

Now we are in a position to prove \autoref{tactician+_lem}.

\begin{proof}[Proof of \autoref{tactician+_lem}.]
Suppose for a contradiction that there are an arena $A=(M,N,E,F,e)$ and challengers 
$\gamma_{M_-}$, $\gamma_{N_+}$ and $\gamma_{\top^*}$
for which \autoref{tactician+_lem} is false.
We pick these such that the set $F$ of upper edges is of minimal size.
Although we will not need it, it is worth noting  that $F$ is nonempty since $\{M_-,N_+,\top^*\}$ is blocking.
We abbreviate $\overline\gamma[f]=\overline \gamma_{M_-}[f]\cup\overline \gamma_{N_+}[f]\cup\overline \gamma_{\top^*}[f]$.

\begin{sublem}\label{every_f_is_blocking}
For each $f\in F$ the set $\overline\gamma[f]$ is blocking.
\end{sublem}

\begin{proof}
This is immediate by the minimality of $|F|$ and \autoref{preblock} and \autoref{rem_challenger}.

\end{proof}

\begin{sublem}\label{567}
For each $f\in F$ one of the following three conditions from \autoref{lem5_minus} is true.
\begin{enumerate}
\setcounter{enumi}{4}
\item $\overline\gamma_{M_-}[f]\se \{M_+,N_+,\top\}$ and 
$\overline\gamma_{\top^*}[f]\se \{M_+^*,N_+^*, \top^*\}$ and 
one of $\{M_-,N_+\}\se \overline\gamma[f]$ or
$\{M_+,N_-\}\se \overline\gamma[f]$;

\item $\overline\gamma_{M_-}[f]=\emptyset$ and
$N_-^*\not\in \overline\gamma_{\top^*}[f]$ and $\{M_-^*,N_+^*,\top\}\se\overline\gamma[f]$
and $\overline\gamma_{N_+}[f]\se \{N_+,\top\}$;

\item $\overline\gamma_{M_-}[f]=\emptyset$ and
$\{M_+^*,N_-^*,\top\}\se\overline\gamma[f]$
and $\overline\gamma_{N_+}[f]\se \{M_+,\top\}$;

\end{enumerate} 
\end{sublem}

\begin{proof}
By \autoref{every_f_is_blocking} and \autoref{lem5_minus} the sets 
$\overline \gamma_{M_-}[f]$, $\overline \gamma_{N_+}[f]$ and $\overline \gamma_{\top^*}[f]$
fulfil one of the outcomes of \autoref{lem5_minus}.
If they satisfy 5,6 or 7 we are done. Otherwise they satisfy one of the conditions 1-4 of 
\autoref{lem5_minus}.

\paragraph{Case 1: $\overline \gamma_{M_-}[f]$, $\overline \gamma_{N_+}[f]$ and $\overline \gamma_{\top^*}[f]$ satisfy 1:}

Let $B$ be one of $\{M_+,M_-^*\}$, $\{M_-,M_+^*\}$, $\{N_+,N_-^*\}$ or $\{N_-,N_+^*\}$ such that $B\se \overline \gamma[f]$.
Then we pick a tactician$^+$ $\mu_B$ as in \autoref{tactician+_exists}.
If $B=\{M_+,M_-^*\}$ or $B=\{M_-,M_+^*\}$, we get the Weak Challenge case in the Packing game. 
If $B=\{N_-,N_+^*\}$ or $B=\{N_+,N_-^*\}$, we get the Weak Challenge case in the Covering game. Thus we get a contradiction in this case.

\paragraph{Case 2: $\overline \gamma_{M_-}[f]$, $\overline \gamma_{N_+}[f]$ and $\overline \gamma_{\top^*}[f]$ satisfy 2:}
Then $M_- \in \overline \gamma_{M_-}[f]$ and $\{N_+, \top^*\} \se \overline \gamma[f]$.
We let $B=\{M_-,N_+,\top^*\}$.
If $N_+\in \overline \gamma_{N_+}[f]$, then
we can define some $\mu$ as in the Double Extension case.
Otherwise,  $N_+\in \overline \gamma_{M_-}[f]$, so that without loss of generality, the $\mu_B$ from \autoref{tactician+_exists} is such that $\pi_1\mu_B(N_+)=M_-$.
Thus
$\mu_B^{-1}(N_+,K)=M_-$ for every tactic $K$ where this is defined.
So we get the  Weak Challenge case in the Packing game.
Thus we get a contradiction in this case.

\paragraph{Case 3: $\overline \gamma_{M_-}[f]$, $\overline \gamma_{N_+}[f]$ and $\overline \gamma_{\top^*}[f]$ satisfy 3:}

Then $B=\{M_+,N_-,\top^*\}\se \overline \gamma[f]$.
Furthermore, there is some tactic $K_1$ fulfilling $M_-$ at $e$ such that $\varphi_{K_1}(f)= N_-$, and some tactic $K_2$ fulfilling $M_-$ or $N_+$ at $e$ with 
$\varphi_{K_2}(f)=M_+$.
So that without loss of generality the $\mu_B$ from \autoref{tactician+_exists} 
is such that $\mu_B(N_-)=(M_-,K_1)$ and one of $\mu_B(M_+)=(M_-,K_2)$ or $\mu_B(M_+)=(N_+,K_2)$.
In particular, for every tactic $K$ the set $\mu_B^{-1}(N_+,K)$ is either empty or is the singleton $\{M_+\}$. So we get the  Weak Challenge case in the Packing game.
Thus we get a contradiction in this case.

\paragraph{Case 4: $\overline \gamma_{M_-}[f]$, $\overline \gamma_{N_+}[f]$ and $\overline \gamma_{\top^*}[f]$ satisfy 4:}

Then there is some tactic $K$ fulfilling $M_-$ at $e$ with $\varphi_K(f)=\top$, and 
either $\{M_-^*,N_+^*\}\se \overline\gamma_{\top^*}[f]$ or $\{M_+^*,N_-^*\}\se \overline\gamma_{\top^*}[f]$. 
If $\{M_-^*,N_+^*\}\se \overline\gamma_{\top^*}[f]$, then we let $B=\{M_-^*,N_+^*,\top\}$.
So that without loss of generality the $\mu_B$ from \autoref{tactician+_exists} is such that
$\mu_B(\top)=(M_-,K)$.
So we get the  Weak Challenge case in the Packing game.
The case $\{M_+^*,N_-^*\}\se \overline\gamma_{\top^*}[f]$ is similar.
Thus we get a contradiction in this case.
\end{proof}

\autoref{567} motivates the following definition:
Let $G\se F$ be the set of those $f\in F$ that satisfy 5 and 
let $H\se F$ be the set of those $f\in F\sm G$ that satisfy 6. Finally
let $J=F\sm G\sm J$. Note that any $f\in J$ satisfies 7 by \autoref{567}.
Now we apply \autoref{lemma27} to $G$, $H$ and $J$. According to which outcome we get, we now split into cases.

\paragraph{Case 1: We get outcome 1 of \autoref{lemma27}:}
There is a wave with $e$ on the $N$-side in $(M/(H\cup J), N/(H\cup J))$.
This wave gives rise to a tactic $K$ fulfilling $M_-$ at $e$ such that:
\[
 \varphi_K(f)=\begin{cases}
   \top \text{ if } f\in H\cup J \\
    M_- \text{ or } N_-  \text{ or } \bot \text{ if } f\in G\\
     \end{cases} 
\]
As $\gamma_{M_-}$ is a challenger, there is some $f\in F$ such that $\gamma_{M_-}(K)=f$.
As $\overline\gamma_{M_-}[x]=\emptyset$ for each $x\in H\cup J$, $f$ cannot be in $H\cup J$ and it cannot be in $G$ either since $\overline\gamma_{M_-}[x]\se\{M_+,N_+,\top\}$ for each $x\in G$.
This is the desired contradiction.

\paragraph{Case 2: We get outcome 2 of \autoref{lemma27}:}
There is a wave $N$-spanning $e$ in $(M\sm (G\cup J),N\sm G/J)$.
This wave gives rise to a tactic $K$ fulfilling $N_+$ at $e$ such that:
\[
 \varphi_K(f)=\begin{cases}
\bot \text{ if } f\in G\\
    M_- \text{ or } N_-  \text{ or } \bot \text{ if } f\in H\\
   N_+ \text{ if } f\in J \\
     \end{cases} 
\]

As $\gamma_{N_+}$ is a challenger, there is some $f\in F$ such that $\gamma_{N_+}(K)=f$.
Note that $f\notin G$. 
As $\overline\gamma_{N_+}[x]\se\{N_+,\top\}$ for each $x\in H$, $f$ cannot be in $H$ and it cannot be in $J$ either since $\overline\gamma_{N_+}[x]\se\{M_+,\top\}$ for each $x\in J$. This is the desired contradiction.

\paragraph{Case 3: We get outcome 3 of \autoref{lemma27}:}
There is some $G'\se G$ and a cohindrance $(Y,T^M,T^N)$ focusing on $e$ in 
$(M\sm (G'\cup J),N\sm G'/ J)$
such that there is some $M$-cocircuit $b$ with $e\in b\se (T^M+e)\sm H$. 
This cohindrance gives rise to a tactic $K$ fulfilling $\top^*$ at $e$ with $C_K^{M^*}=b$ such that:
\[
 \varphi_K(f)=\begin{cases}
\top^* \text{ if } f\in G' \\
M_-^* \text{ or } N_-^*  \text{ or } \bot^* \text{ if } f\in H\cup (G\sm G') \\
M_+^* \text{ if } f\in J \\
       \end{cases} 
\]
Let $f=\gamma_{\top^*}(K)$.
If $f\in G$, 
then it is in $G'$ because $\overline\gamma_{\top^*}[x]\se\{M_+^*,N_+^*,\top^*\}$ for each $x\in G$.
Then 
we let $B=\{M_-,N_+,\top^*\}$ (if $\{M_-,N_+\}\se \overline \gamma[f]$) or $B=\{M_+,N_-,\top^*\}$ (if $\{M_+,N_-\}\se \overline \gamma[f]$). We pick $\mu_B$ such 
that $\mu_B(\top^*)=(\top^*,K)$. 
Thus we get the Weak Challenge case in the Covering game as  $C_K^{M^*}$ does not meet $G'$.

If $f\in H$, then $\varphi_K(f)=M_-^*$ as $N_-^*\not\in \overline\gamma_{\top^*}[f]$ 
and $\gamma_{\top^*}$ is a challenger. Thus we let $B=\{M_-^*,N_+^*,\top\}$ and we pick $\mu_B$ such 
that $\mu_B(M_-^*)=(\top^*,K)$. Thus we get the Weak Challenge case in the Covering game.

If $f\in J$, we let $B=\{M_+^*,N_-^*,\top\}$ and we pick $\mu_B$ such 
that $\mu_B(M_+^*)=(\top^*,K)$. Thus we get the Weak Challenge case in the Covering game.
This completes the proof.
\end{proof}

\subsection{Proof of \autoref{tactician-_lem}}

The aim of this subsection is to prove \autoref{tactician-_lem}.
The structure of the proof will be as in the last subsection.

\begin{lem}\label{lemma17}
 Let $M$ and $N$ be two matroids on the same finite ground set $E$.
Let $H,J\se E$ disjoint  and $e\in E\sm H\sm J$.
Then one of the following is true.
\begin{enumerate}
 \item There is some $H'\se H$ such that there is a wave $M$-spanning $e$ in $(M/H'/J, N\sm H'/J)$.
\item There is some $J'\se J$ and a wave $(X,S^M,S^N)$ with $e$ on the $M$-side in $(M/J',N/J')$ such that there is an $N$-circuit $o$ with $e\in o\se(S^N+e)\sm H$. 
\item There is some $H'\se H$ and a cohindrance $(Y,T^M,T^N)$ focusing on $e$ in $(M/H',N\sm H')$
such that there is some $M$-cocircuit $b$ with $e\in b\se (T^M+e)\sm J$. 
\end{enumerate}

\end{lem}

If $H=J=\emptyset$, then this lemma just says that $\{M_+,N_-,\top^*\}$ is blocking.
So this lemma can be seen as an extension of this fact.

\begin{proof}
We prove this lemma by induction on the size of $E$.

\paragraph{Case 1: $H=\emptyset$.}
\paragraph{Subcase 1.1: There is a nonempty wave $(Z,U^M,U^N)$ in $(M,N)$ avoiding $e$.}
Now we apply the induction hypothesis to $(M/Z,N/Z)$ and $\emptyset$ and $J\sm Z$.
If we have outcome 3, we immediately get outcome 3 in $(M,N)$.
If we get a wave as in outcome 1 or 2, we stick $(Z,U^M,U^N)$ onto that wave by \autoref{joinwaves} and get outcome 1 or 2, respectively, in $(M,N)$.

To complete Case 1, it remains to consider the following subcase.

\paragraph{Subcase 1.2: Every nonempty wave in $(M,N)$ contains $e$.} If $J=\emptyset$, we can just use the fact that $\{M_+,N_-,\top^*\}$ is blocking by \autoref{blockstr}.
So let $j\in J$. 
Now we apply the induction hypothesis to $(M/j,N\sm j)$.
If we get outcome 1 or 2, we get outcome 1 or 2, respectively, in  $(M,N)$.
In particular, there is no hindrance in $(M/j,N\sm j)$ focusing on $e$.
So by \autoref{intermediate_lemma7}, $E-j$ is a cowave in $(M/j,N\sm j)$.
And we may assume that we get a cohindrance $(Y,T^M,T^N)$ 
and an $M$-cocircuit as in outcome 3.
By sticking $E-j$ onto $(Y,T^M,T^N)$ if necessary, we may assume that 
$Y=E-j$. 
The edge $j$ is not an $M$-loop since otherwise $\{j\}$ would be a nonempty wave not containing $e$, contrary to our assumption.
Thus $(E,S^M,S^N+j)$ is a cohindrance in $(M,N)$, which together with $b$ witnesses outcome 3.

\paragraph{Case 2: $H\neq \emptyset$.}
\paragraph{Subcase 2.1: There is a nonempty cowave $(Z,U^M,U^N)$ in $M/J,N/J$ avoiding $e$.}
Now we apply the induction hypothesis to $(M\sm Z,N\sm Z)$ and $H\sm Z$ and $J$.
Just as in Subcase 1.1, one checks that if one gets outcome 1, 2 or 3 in the minor, then one gets outcome 1, 2 or 3 in $(M,N)$, respectively.

Thus it remains to consider the following subcase:

\paragraph{Subcase 2.2: Every cowave in $(M/J,N/J)$ contains $e$.}
As we are in Case 2, there is some $h\in H$, and we apply the induction hypothesis to
$(M/h,N\sm h)$ and $H-h$ and $J$.
If we get outcome 1 or 3, we get outcome 1 or 3, respectively, in $(M,N)$.
So we may assume that we have outcome 2:
There is some $J'\se J$ and a wave $(X,S^M,S^N)$ with $e$ on the $M$-side in $(M/(J'+h),N/J'\sm h)$ such that there is an $N$-circuit $o$ with 
$e\in o\se(S^N+e)\sm (H-h)=(S^N+e)\sm H$. 
By adding the edges in $J\sm (X\cup J')$ to $J'$ if necessary, we may assume that
$J'=J\sm X$.

The wave $(X,S^M,S^N)$ is almost the wave we are looking for, except that
it lives in the wrong pair of matroids. Next, we shall extend $(X,S^M,S^N)$
to a wave living in the right pair of matroids.
By the dual of \autoref{intermediate_lemma7},
 either there is a cohindrance focusing on $e$ in $(M/(J+h),N/J\sm h)$
or $E\sm (J+h)$ is a wave in $(M/(J+h),N/J\sm h)$ with spanning sets $T^M$ and $T^N$.
We may assume that the second occurs since the first gives us outcome 3.
Then $(E\sm (J\cup X+h), T^M\sm X, T^N\sm X)$ is a wave in 
$(M/(J\cup X+h),N/(J\cup X)\sm h)$.
By sticking $X$ onto that wave, we get that
$(E\sm (J'+h), S^M\cup T^M\sm X, S^N\cup T^N\sm X)$ is a wave in $(M/(J'+h),N/J'\sm h)$.

The edge $h$ is not an $N$-coloop since otherwise $\{h\}$ would be a cowave, contrary to our assumption in this subcase. So $h$ is not an $N/J'$-coloop either.
Thus $(E\sm J', (S^M+h)\cup T^M\sm X, S^N\cup T^N\sm X)$ is a wave in $(M/J',N/J')$
with $e$ on the $M$-side. Moreover the circuit $o$ witnesses that we have outcome 2.
\end{proof}

\autoref{lemma17} is the main principle we use in the proof of \autoref{tactician-_lem}.
The work of bridging from \autoref{lemma17} to \autoref{tactician-_lem} is done in the following lemma.

\begin{lem}\label{lem4-}
Let $\overline \gamma_{M_+}$, $\overline\gamma_{N_-}\se \Pcal-\bot$ and $\overline\gamma_{\top^*}\se \Pcal^*-\bot^*$
be up-closed such that $\overline \gamma=\overline\gamma_{M_+}\cup\overline\gamma_{N_-}\cup\overline\gamma_{\top^*}$ is blocking.
Then one of the following is true. 
\begin{enumerate}
 \item One of the 6 sets $\{M_+,M_-^*\}$, $\{M_-,M_+^*\}$, $\{N_+,N_-^*\}$, $\{N_-,N_+^*\}$,  
$\{M_-,N_+,\top^*\}$ or $\{M_+^*,N_-^*,\top\}$
is a subset of $\overline \gamma$;
\item $\{M_+,N_-,\top^*\}\se \overline\gamma$ and $\overline\gamma_{M_+}$ meets $\{M_+,N_-\}$;
\item $\top\in \overline\gamma_{M_+}$ and $\{M_-^*,N_+^*\}\se \overline\gamma_{\top^*}$;
\item $\overline\gamma_{M_+}\se \{\top,N_+\}$ and 
$\overline\gamma_{N_-}= \{\top,M_+,N_+,N_-\}$ and
$\top^*\in \overline\gamma_{\top^*}\se \{\top^*,M_+^*\}$;
\item $\overline\gamma_{M_+}=\emptyset$ and 
$\top\in \overline\gamma_{N_-}\se \{\top,N_+\}$ and
$\overline\gamma_{\top^*}=\{\top^*,M_+^*,M_-^*,N_+^*\}$;
\end{enumerate}
\end{lem}

\begin{proof}
Since $\overline\gamma\cap(\Pcal+\bot^*)$ is nonempty and $\bot^*\notin \overline\gamma$, 
we get $\overline\gamma\cap\Pcal$ is nonempty, thus $\top\in \overline\gamma$. 
Similarly, $\top^*\in \overline\gamma$.
Now suppose for a contradiction that we do not have one of the outcomes 1-5.
By \autoref{blockstr}, 
either $\{M_+,N_-,\top^*\}\se \overline \gamma$ or $\{M_-^*,N_+^*,\top\}\se \overline \gamma$
as we do not have outcome 1.

If $\{M_+,N_-,\top^*\}\se \overline \gamma$, then $N_+\in \overline \gamma$.
So
$\overline \gamma$ contains neither $M_-^*$ nor $N_+^*$ nor $M_-$ as we do not have outcome 1.
Since we do not have outcome 2, it must be that $\overline \gamma_{M_+}$ avoids $\{M_+,N_-\}$.
Thus $\overline \gamma_{M_+}\se \{\top,N_+\}$ and so 
$\overline\gamma_{N_-}= \{\top,M_+,N_+,N_-\}$ and $\top^*\in \overline\gamma_{\top^*}\se \{\top^*,M_+^*\}$, as they are both up-closed.
Thus we have outcome 4,
which is a contradiction in this case.

Hence it suffices to consider the case that $\{M_-^*,N_+^*,\top\}\se \overline \gamma$.
Then $M_+^*\in \overline \gamma$.
So $\overline \gamma$ contains neither $M_+$ nor $N_-$ nor $N_-^*$ as we do not have outcome 1.
Since we do not have outcome 3, $\overline\gamma_{M_+}=\emptyset$.
Thus $\top\in \overline\gamma_{N_-}\se \{\top,N_+\}$ and
$\overline\gamma_{\top^*}=\{\top^*,M_+^*,M_-^*,N_+^*\}$ as they are both up-closed. 
Thus we have outcome 5, which is the desired contradiction.
\end{proof}

Now we are in a position to prove \autoref{tactician-_lem}.

\begin{proof}[Proof of \autoref{tactician-_lem}.]
Suppose for a contradiction that there are an arena $A=(M,N,E,F,e)$ and challengers 
$\gamma_{M_+}$, $\gamma_{N_-}$ and $\gamma_{\top^*}$
for which \autoref{tactician-_lem} is false.
We pick these such that the set $F$ of upper edges is of minimal size.
Although we will not need it, it is worth noting that $F$ is nonempty since $\{M_+,N_-,\top^*\}$ is blocking.
We abbreviate $\overline\gamma[f]=\overline \gamma_{M_+}[f]\cup\overline \gamma_{N_-}[f]\cup\overline \gamma_{\top^*}[f]$.
The following sublemma may be proved in a similar way to \autoref{every_f_is_blocking}.

\begin{sublem}\label{every_f_is_blocking-}
For each $f\in F$ the set $\overline\gamma[f]$ is blocking.
\end{sublem}

\begin{sublem}\label{4or5}
For each $f\in F$ one of the following two conditions from \autoref{lem4-} is true.
\begin{enumerate}
\setcounter{enumi}{3}
\item $\overline\gamma_{M_+}[f]\se \{\top,N_+\}$ and 
$\overline\gamma_{N_-}[f]= \{\top,M_+,N_+,N_-\}$ and
$\top^*\in \overline\gamma_{\top^*}[f]\se \{\top^*,M_+^*\}$;
\item $\overline\gamma_{M_+}[f]=\emptyset$ and 
$\top\in \overline\gamma_{N_-}[f]\se \{\top,N_+\}$ and
$\overline\gamma_{\top^*}[f]=\{\top^*,M_+^*,M_-^*,N_+^*\}$;
\end{enumerate} 
\end{sublem}

\begin{proof}
By \autoref{every_f_is_blocking-} and \autoref{lem4-} the sets 
$\overline \gamma_{M_+}[f]$, $\overline \gamma_{N_-}[f]$ and $\overline \gamma_{\top^*}[f]$
fulfil one of the outcomes of \autoref{lem4-}.
If they satisfy 4 or 5, we are done. Otherwise they satisfy one of the conditions 1-3 of 
\autoref{lem4-}.

First suppose for a contradiction that they satisfy 1:
Let $B$ be one of $\{M_+,M_-^*\}$, $\{M_-,M_+^*\}$, $\{N_+,N_-^*\}$, $\{N_-,N_+^*\}$,  
$\{M_-,N_+,\top^*\}$ or $\{\top,N_-^*,M_+^*\}$ such that $B\se \overline \gamma[f]$.
Then we pick a tactician$^-$ $\mu_B$ as in \autoref{tactician-_exists}.
If $B=\{M_-,N_+,\top^*\}$, we get Improvement case 1.
If  $B=\{\top,N_-^*,M_+^*\}$, we get Improvement case 2.
If $B=\{M_+,M_-^*\}$ or $B=\{M_-,M_+^*\}$, we get the Weak Challenge case in the Packing game. 
If $B=\{N_-,N_+^*\}$ or $B=\{N_+,N_-^*\}$, we get the Weak Challenge case in the Covering game. Thus we get a contradiction in this case.

Next, we consider the case that $\overline \gamma_{M_+}[f]$, $\overline \gamma_{N_-}[f]$ and $\overline \gamma_{\top^*}[f]$ satisfy 2:
$\{M_+,N_-,\top^*\}\se \overline\gamma[f]$ and $\overline\gamma_{M_+}[f]$ meets $\{M_+,N_-\}$.
We let $B=\{M_+,N_-,\top^*\}$.
If $\overline\gamma_{M_+}[f]\cap\{M_+,N_-\}=\{M_+\}$, 
then $N_-\in \overline\gamma_{N_-}[f]$, and so
we can define some $\mu$ as in the Double Extension case.
Otherwise $N_-\in \overline\gamma_{M_+}[f]$, so that without loss of generality $\mu_B$ is such that $\pi_1(\mu_B(N_-))=M_+$.
Thus
$\mu_B^{-1}(N_-,K)=M_+$ for every tactic $K$ where this is defined.
So we get the  Weak Challenge case in the Packing game.

Thus, it remains to consider the case that $\overline \gamma_{M_+}[f]$, $\overline \gamma_{N_-}[f]$ and $\overline \gamma_{\top^*}[f]$ satisfy 3:
$\top\in \overline\gamma_{M_+}[f]$ and $\{M_-^*,N_+^*\}\se \overline\gamma_{\top^*}[f]$.
We let $B=\{\top, M_-^*,N_+^*\}$ and define the tactician$^-$ $\mu_B$
such that $\pi_1(\mu_B(\top))=M_+$.
So we have the Weak Challenge case in the Packing game. This completes the proof.
\end{proof}

\autoref{4or5} motivates the following definition:
Let $H\se F$ be the set of those $f\in F$ that satisfy 4 and let $J=F\sm H$. Note that any $j\in J$ satisfies 5 by \autoref{4or5}.
Now we apply \autoref{lemma17} to $J$ and $H$. According to which outcome we get, we now split into cases.

\paragraph{Case 1: We get outcome 1 of \autoref{lemma17}:}
There is some $H'\se H$ such that there is a wave $M$-spanning $e$ in $(M/H'/J, N\sm H'/J)$.
This wave gives rise to a tactic $K$ fulfilling $M_+$ at $e$ such that:
\[
 \varphi_K(f)=\begin{cases}
    M_+ \text{ if } f\in H' \\
    M_- \text{ or } N_-  \text{ or } \bot \text{ if } f\in H\sm H' \\
    \top\text{ if } f\in J \\
  \end{cases} 
\]
As $\gamma_{M_+}$ is a challenger, there is some $f\in F$ such that $\gamma_{M_+}(K)=f$.
As $\overline\gamma_{M_+}[j]=\emptyset$ for each $j\in J$, $f$ cannot be in $J$ and it cannot be in $H$ either since $\overline\gamma_{M_+}[h]\se\{\top,N_+\}$ for each $h\in H$.
This is the desired contradiction.

\paragraph{Case 2: We get outcome 2 of \autoref{lemma17}:}
There is some $J'\se J$ and a wave $(X,S^M,S^N)$ with $e$ on the $M$-side in $(M/J',N/J')$ such that there is an $N$-circuit $o$ with $e\in o\se(S^N+e)\sm H$. 
This wave gives rise to a tactic $K$ fulfilling $N_-$ at $e$ with $C_K^N=o$ such that:
\[
 \varphi_K(f)=\begin{cases}
\top \text{ if } f\in J' \\
       M_- \text{ or } N_-  \text{ or } \bot \text{ if } f\in H\cup (J\sm J') \\
       \end{cases} 
\]
Let $f=\gamma_{N_-}(K)$.
If $f\in H$, we let $B=\{M_+,N_-,\top^*\}$ and we pick $\mu_B$ such 
that $\mu_B(N_-)=(N_-,K)$. Thus we get the Weak Challenge case in the Packing game.
If $f\in J$, 
then it must be in $J'$ because $\overline\gamma_{N_-}[j]\se\{\top,N_+\}$ for each $j\in J$.
Then we let $B=\{M_-^*,N_+^*,\top\}$ and we pick $\mu_B$ such 
that $\mu_B(\top)=(N_-,K)$. Thus we get the Weak Challenge case in the Packing game.

\paragraph{Case 3: We get outcome 3 of \autoref{lemma17}:}
There is some $H'\se H$ and a cohindrance $(Y,T^M,T^N)$ focusing on $e$ in $(M/H',N\sm H')$
such that there is some $M$-cocircuit $b$ with $e\in b\se (T^M+e)\sm J$. 
This cohindrance gives rise to a tactic $K$ fulfilling $\top^*$ at $e$ with $C_K^{M^*}=b$ such that:
\[
 \varphi_K(f)=\begin{cases}
N_+^* \text{ if } f\in H' \\
       M_-^* \text{ or } N_-^*  \text{ or } \bot^* \text{ if } f\in J\cup (H\sm H') \\
       \end{cases} 
\]
Let $f=\gamma_{\top^*}(K)$.
If $f\in H$, 
then it is in $H'$ because $\overline\gamma_{\top^*}[h]\se\{\top^*,M_+^*\}$ for each $h\in H$.
Then 
we let $B=\{M_+,N_-,\top^*\}$ and we pick $\mu_B$ such 
that $\mu_B(\top^*)=(\top^*,K)$. Thus we get the Weak Challenge case in the Covering game.
If $f\in J$, we let $B=\{M_-^*,N_+^*,\top\}$ and we pick $\mu_B$ such 
that $\mu_B(M_-^*)=(\top^*,K)$. Thus we get the Weak Challenge case in the Covering game.

\end{proof}

\section{Concluding remarks}
There does not seem to be any reason, in principle, why the methods of this paper should not extend to trees of matroids with larger overlap. However, we have found that the most naive attempt to do this results in an explosion of the number of cases which must be dealt with.
This puts the necessary computations far beyond the bounds of what we could reasonably attempt.

It is clear that the success of our argument provides only weak evidence for the truth of the Packing/Covering conjecture in general. However, we think further evidence for this conjecture is given by the fact that our argument only just succeeds: an argument which resolved this case more straightforwardly would have suggested that the conclusion was an artifact of the tree structure, rather than relying on the matroidal structure.

There are some subtle issues of descriptive set theory surrounding the question of how much we have shown. If we have matroids $M_{\Psi_M}(T, M)$ and $M_{\Psi_N}(T, N)$ which are matroids {\em because} the sets $\Psi_M$ and $\Psi_N$ are Borel then our results allow us to deduce that $(M_{\Psi_M}(T, M), M_{\Psi_N}(T, N))$ satisfies Packing/Covering. The same comment applies if $\Psi_M$ and $\Psi_N$ are taken from some other class of determined sets closed under basic operations such as inverse images under continuous functions, for example the class of analytic sets if there is a measurable cardinal. But if $M_{\Psi_M}(T, M)$ and $M_{\Psi_N}(T, N)$ just happen to be matroids then it is not clear that they must satisfy Packing/Covering, because the games on whose determinacy our argument relies are quite different from the games whose determinacy witnesses that $M_{\Psi_M}(T, M)$ and $M_{\Psi_N}(T, N)$ are matroids.

If it could be shown that the determinacy of the latter collection of games implies that of the former, then this would significantly strengthen our result. On the other hand, a counterexample to this implication would give a counterexample to the Packing/Covering conjecture. Although for these reasons we are eager to resolve this issue, only the bare beginnings of an investigation into questions like this (regarding the relationship of determinacy of particular games outside well-behaved classes) has been made, in papers such as \cite{MR2178987} and \cite{MR2540945}.

\bibliographystyle{plain}
\bibliography{literatur}
\end{document}